\documentclass[11pt]{article}
 \pdfoutput=1                           
\usepackage{a4wide}

\usepackage{amsmath,amsfonts,amssymb,amsthm}
\usepackage{graphics}
\usepackage{ifpdf}

\usepackage{color}
\usepackage{hyperref}

\definecolor{modra3}{rgb}{.1,.0,.4}
\hypersetup{colorlinks=true, linkcolor=modra3, urlcolor=modra3, citecolor=modra3, pdfpagemode=UseNone, pdfstartview=}

\newtheorem{theorem}{Theorem}[section]
\newtheorem{proposition}[theorem]{Proposition}
\newtheorem{observation}[theorem]{Observation}
\newtheorem{lemma}[theorem]{Lemma}
\newtheorem{question}[theorem]{Question}

\newtheorem{corollary}[theorem]{Corollary}

\theoremstyle{remark}
\newtheorem{definition}[theorem]{Definition}
\newtheorem*{remark}{Remark}

\newcommand{\ex}{\mathrm{ex}}
\newcommand{\exminor}{\mathrm{exm}}
\newcommand{\row}{r}
\newcommand{\hust}{q}
\newcommand{\hloubka}{\tilde{h}}
\newcommand{\Cross}{\mathrm{Cross}}

\begin{document}

\title{Better upper bounds on the F\"{u}redi--Hajnal limits of permutations\thanks{An extended abstract of this paper appeared in Proceedings of the Twenty-Eighth Annual ACM-SIAM Symposium on Discrete Algorithms (SODA 2017), 2280--2293. The second author was supported by ERC Advanced Research Grant no 267165 (DISCONV), by the grant no.
18-19158S of the Czech Science Foundation (GA\v{C}R) 
and by Charles University project UNCE/SCI/004.}} 

\author{Josef Cibulka\footnote{Department of Applied Mathematics and Institute for Theoretical Computer Science, Charles University, Faculty of Mathematics and Physics, Malostransk\'e n\'am.~25, 118~00~ Praha 1, Czech Republic; \texttt{cibulka@kam.mff.cuni.cz, kyncl@kam.mff.cuni.cz}
} 
\and 
Jan Kyn\v{c}l\footnotemark[2] \thanks{Alfr\'ed R\'enyi Institute of Mathematics, Re\'altanoda u. 13-15, Budapest 1053, Hungary}
}

\date{}

\maketitle

\begin{abstract}
A \emph{binary matrix} is a matrix with entries from the set $\{0,1\}$.
We say that a binary matrix $A$ \emph{contains} a binary matrix $S$ if $S$ can
be obtained from $A$ by removal of some rows, some columns, and changing some
$1$-entries to $0$-entries. If $A$ does not contain $S$, we say that $A$ 
\emph{avoids} $S$. A \emph{$k$-permutation matrix} $P$ is a binary $k \times k$ matrix
with exactly one $1$-entry in every row and one $1$-entry in every column.

The F\"{u}redi--Hajnal conjecture, proved by Marcus and Tardos, states
that for every permutation matrix $P$, there is a constant $c_P$ such that for
every $n \in \mathbb{N}$, every $n \times n$ binary matrix $A$ with at least $c_P n$
$1$-entries contains $P$. 

We show that $c_P \le 2^{O(k^{2/3}\log^{7/3}k / (\log\log k)^{1/3})}$ 
asymptotically almost surely for a random $k$-permutation matrix $P$. 
We also show that $c_P \le 2^{(4+o(1))k}$ for every $k$-permutation matrix $P$, 
improving the constant in the exponent of a recent upper bound on $c_P$ by Fox.
Moreover, we improve the upper bound on $c_P$ in terms of the Stanley--Wilf limit $s_P$ to $c_P \le O\big(s_P^{2.75} \log s_P\big)$.

We also consider a higher-dimensional generalization of the Stanley--Wilf conjecture about the number of $d$-dimensional $n$-permutation matrices avoiding a fixed $d$-dimensional $k$-permutation matrix,
and prove almost matching upper and lower bounds of the form $(2^k)^{O_d(n)} \cdot (n!)^{d-1-1/(d-1)}$
and $n^{-O_d(k)} k^{\Omega_d(n)} \cdot (n!)^{d-1-1/(d-1)}$, respectively.
\end{abstract}


\section{Introduction}
A \emph{binary matrix} is a matrix with entries from the set $\{0,1\}$.
We say that an $n \times n$ binary matrix $A$ \emph{contains} a $k \times k$ binary
matrix $B$ if $B$ can be obtained from $A$ by removing some rows, some columns
and by changing some $1$-entries to $0$-entries.
If $A$ does not contain $B$, we say that $A$ \emph{avoids} $B$.

For every $n \in \mathbb{N}$, we abbreviate the set $\{1, 2, \dots, n\}$ as $[n]$.
A \emph{$k$-permutation} $\pi$ is a permutation on $[k]$, that is, 
a bijective function $\pi: [k] \rightarrow [k]$.
We will also sometimes represent a permutation by the sequence of the function 
values, that is, as $(\pi(1), \pi(2), \dots, \pi(k))$.
A \emph{permutation matrix} is a square binary matrix with exactly one $1$-entry 
in every row and in every column.
A $k \times k$ permutation matrix is also called a \emph{$k$-permutation matrix}.
A $k$-permutation matrix $P$ corresponds to the $k$-permutation $\pi$ 
satisfying, for every $i,j \in [k]$, $\pi(i)=j$ if and only if $P_{i,j} = 1$.
Note that by this definition, a graph of $\pi$ as a function is 
obtained by rotating $P$ by $90$ degrees counterclockwise.

The \emph{restriction} of an $n$-permutation $\rho$ on a set $\{s_1, s_2, \dots, s_l\}$ of 
positions, where $1 \le s_1 < s_2 < \dots < s_l \le n$, is the $l$-permutation $\pi$ 
where $\pi(i) < \pi(j)$ if and only if $\rho(s_i) < \rho(s_j)$ for every 
$i,j \in [l]$.
If $\pi$ is not a restriction of $\rho$ on any set of positions, then we say
that $\rho$ \emph{avoids} $\pi$.
By definition, a permutation $\pi$ is a restriction of a permutation $\rho$ if and only
if the permutation matrix $Q$ corresponding to $\rho$ contains the permutation matrix 
$P$ corresponding to $\pi$.

For a binary matrix $A$ and $n \in \mathbb{N}$, let $\ex_A(n)$ be the maximum 
number of $1$-entries in an $n \times n$ binary matrix avoiding $A$.
The F\"{u}redi--Hajnal conjecture~\cite{FurediHajnal}, proved by Marcus and 
Tardos~\cite{MarcusTardos}, states that for every permutation matrix $P$, there is 
a constant $c$ such that for every $n \in \mathbb{N}$, we have $\ex_P(n) \le c n$. 

For a permutation matrix $P$ and $n \in \mathbb{N}$, let $S_P(n)$ be the number 
of $n$-permutation matrices that avoid $P$.
In other words, $S_P(n)$ is the number of $n$-permutations avoiding $\pi$, where 
$\pi$ is the permutation corresponding to $P$.
The Stanley--Wilf conjecture states that for every permutation matrix $P$, there
is a constant $s$ such that $|S_P(n)| \le s^n$ for every $n \in \mathbb{N}$.
The validity of the conjecture follows from the validity of the F\"{u}redi--Hajnal 
conjecture by an earlier result of Klazar~\cite{Klazar00}.

Fix a permutation matrix $P$.
Arratia~\cite{Arratia99} showed by the supermultiplicativity of the function 
$|S_P(n)|$ that the validity of the Stanley--Wilf conjecture implies that the limit
\[
s_P = \lim_{n \to \infty} |S_P(n)|^{1/n}
\]
exists and is finite.
Similarly, the superadditivity of $\ex_P(n)$~\cite[Lemma 1(ii)]{PachTardos} together with the F\"{u}redi--Hajnal conjecture imply the 
same conclusion for the limit
\[
c_P = \lim_{n \to \infty} \ex_P(n)/n.
\]
The numbers $s_P$ and $c_P$ are called the \emph{Stanley--Wilf limit} and the \emph{F\"{u}redi--Hajnal limit} of $P$, 
respectively.
We will often refer to $s_P$ as the Stanley--Wilf limit of the permutation $\pi$
corresponding to $P$.

The Marcus--Tardos proof~\cite{MarcusTardos} of the F\"{u}redi--Hajnal conjecture 
implies the upper bound $c_P \le 2k^4\binom{k^2}{k}$ for every $k$-permutation matrix $P$.
Klazar's reduction shows that $s_P \le 15^{c_P}$ for every permutation matrix $P$,
thus showing $s_P \le 2^{2^{O(k \log(k))}}$ for every $k$-permutation matrix $P$.
The first author~\cite{Cibulka09} showed that the values of the two limits are 
close to each other, in particular, $\Omega(c_P^{2/9}) \le s_P \le 2.88 c^2_P$
for every permutation matrix $P$.
Fox~\cite{Fox13} improved the upper bound on the F\"{u}redi--Hajnal limit of $k$-permutation matrices 
to $c_P \le 3k 2^{8k}$ (which can be easily lowered 
to $c_P \le k^{O(1)} 2^{6k}$). 
Thus both $s_P$ and $c_P$ are in $2^{O(k)}$, where $k$ is the size of $P$.

The Stanley--Wilf limit of the identity $k$-permutation is $(k-1)^2$~\cite{Regev81}.
By a result of Valtr published in~\cite{KaiserKlazar}, for every $k$ and every 
$k$-permutation matrix $P$, $s_P \ge (k-1)^2 /e^3$. 
Let $s_{1324}$ be the Stanley--Wilf limit of the permutation $(1,3,2,4)$.
Albert et al.~\cite{AERWZ06} proved the lower bound $s_{1324} \ge 9.47$.
B\'{o}na~\cite{Bona07Records} proved that there are infinitely many permutation 
matrices $P$ with $s_P \ge s_{1324}\cdot (k-1)^2/9$. 
Bevan~\cite{Bevan14} increased the lower bound on $s_{1324}$ to $9.81$, 
thus increasing the lower bound for infinitely many permutation matrices $P$ to $s_P \ge 9.81(k-1)^2/9$.
Until recently, no $k$-permutation has been known to have the Stanley--Wilf limit larger 
than quadratic in $k$, which lead to the widely believed conjecture that the Stanley--Wilf 
limits are always at most quadratic in $k$; see the survey by 
Steingr\'{\i}msson~\cite{Steingrimsson13:survey}.
This belief was further supported by the fact that the Stanley--Wilf limit of every 
layered $k$-permutation is bounded from above by $4k^2$~\cite{CJS12} (a \emph{layered permutation} is a concatenation of decreasing sequences $S_1,S_2,\dots,S_l$ such that for every $i\le l-1$, all elements of $S_i$ are smaller than all elements of $S_{i+1}$).

The situation concerning the F\"{u}redi--Hajnal limit was similar. 
A simple observation gives the lower bound $c_P \ge 2(k-1)$ for all $k$-permutation 
matrices and this lower bound is attained by the $k \times k$ unit matrix and several 
other $k$-permutation matrices~\cite{FurediHajnal}.
The best lower bound on the F\"{u}redi--Hajnal limit of some class of permutations
was quadratic in the size of the permutations~\cite{Cibulka09}.

A breakthrough occurred when Fox~\cite{Fox13} gave a randomized construction showing 
that for every $k$, there are $k$-permutation matrices $P$ with 
$c_P \ge 2^{\Omega(k^{1/2})}$ and thus $s_P \ge 2^{\Omega(k^{1/2})}$.
He additionally showed that as $k$ goes to infinity, almost all $k$-permutation
matrices satisfy $c_P \ge 2^{\Omega((k/\log k)^{1/2})}$.

\paragraph{Contractions and interval minors.}
Contracting rows, columns and blocks is a crucial technique for studying permutation avoidance. 
Let $A$ be an $n \times n$ binary matrix with rows $r_1, r_2, \dots, r_n$, in this order.
A partition $\mathcal{I} = \{I_1, I_2, \dots, I_t\}$ of the set of rows of $A$ is called an \emph{interval decomposition} of the rows of $A$ if each of the sets $I_j$ consists of a nonzero number of consecutive rows, and $i<i'$ whenever $j<j'$, $r_i\in I_j$ and $r_{i'}\in I_{j'}$. The sets $I_j$ are called the \emph{intervals} of the decomposition.
An interval decomposition of the columns is defined analogously.

A \emph{block decomposition} of $A$ is determined by a row decomposition $\mathcal{I}=\{I_1, I_2, \dots, I_t\}$ 
and a column decomposition $\mathcal{I}'= \{I'_1, I'_2, \dots, I'_{t'}\}$ as follows.
For every $i \in [t]$ and $j \in [t']$, the \emph{$(i,j)$-block} of $A$ is the 
submatrix of $A$ on the intersection of $I_{i}$ and $I'_{j}$.
To \emph{contract the blocks} means to create a $t \times t'$ binary matrix $B$ such that $B_{i,j}=0$ if and only if the $(i,j)$-block of $A$ contains only zeros.
To \emph{contract by an interval decomposition of rows} means to create a matrix 
with one row for each interval where each row has $0$-entries exactly in those columns where the corresponding interval has only zeros.
\emph{Contraction by an interval decomposition of columns} is defined analogously.

A binary matrix $B$ is an \emph{interval minor} of a binary matrix $A$ if $B$ can be 
obtained from $A$ by the contraction of blocks of some block decomposition followed 
possibly by replacing some $1$-entries with $0$-entries.
Although contractions were used earlier, the interval minors were defined only 
recently by Fox~\cite{Fox13}.

The matrix $J_{r,k}$ is the $r \times k$ matrix with $1$-entries only.
The matrix $J_{k,k}$ is abbreviated as $J_k$.

Given a binary matrix $B$, let $\exminor_B(n)$ be the maximum number of $1$-entries in 
an $n \times n$ matrix $A$ such that $B$ is not an interval minor of $A$.
Clearly, if $P$ is a permutation matrix, then for every binary matrix $A$, $A$
contains $P$ if and only if $P$ is an interval minor of $A$, and thus 
$\exminor_P(n) = \ex_P(n)$.
Furthermore, if $M$ is an interval minor of $A$, then every 
interval minor $B$ of $M$ is also an interval minor of $A$.

Marcus and Tardos~\cite{MarcusTardos} actually proved that $\exminor_{J_k}(n) \le 2k^4\binom{k^2}{k} n$, 
which implies the same upper bound on $\ex_P(n)$ for every $k$-permutation matrix $P$.
Fox~\cite{Fox13} improved the upper bound to $\exminor_{J_k}(n) \le 3k 2^{8k} n$.

\paragraph{Higher-dimensional matrices.}
Similar questions can be asked for higher-dimensional permutation matrices. 

We call $M \in \{0,1\}^{[n_1] \times \dots \times [n_d]}$ a \emph{d-dimensional 
binary matrix of size $n_1 \times \dots \times n_d$}.
A $d$-dimensional binary matrix $P$ of size $k \times \dots \times k$ is a \emph{$d$-dimensional
$k$-permutation matrix} if $P$ contains $k$ 1-entries and the positions of every pair of 1-entries
of $P$ differ in all coordinates. 

We say that a $d$-dimensional binary matrix $P=(p_{i_1,\dots,i_d})$ of size $k_1 \times \dots \times k_d$
is \emph{contained} in a $d$-dimensional binary matrix $A=(a_{i_1,\dots,i_d})$ of size $n_1 \times \dots \times n_d$
if there exist $d$ increasing injections $f_i: [k_i] \rightarrow [n_i]$, $i=1,2,\dots,d$
such that for all $i_1, i_2, \dots, i_d \in [k]$, if $p_{i_1, \dots, i_d} = 1$ then
$a_{f_1(i_1), \dots, f_d(i_d)} = 1$. If $P$ is not contained in $A$, we say that 
$A$ \emph{avoids} $P$. 

When $P$ is a $d$-dimensional permutation matrix, we let $\ex_P(n)$ be the maximum number of $1$-entries in
a $P$-avoiding $n \times \cdots \times n$ $d$-dimensional binary matrix.
Klazar and Marcus~\cite{KlazarMarcus} proved an analogue of the F\"uredi--Hajnal conjecture 
for higher-dimensional matrices. 
For any given $d$-dimensional permutation matrix $P$, they showed that $\ex_P(n) \le 2^{O(k\log k)} n^{d-1}$.
Geneson and Tian~\cite[Equation (4.5)]{GenesonTian} improved the upper bound to $\ex_P(n) \le 2^{O(k)} n^{d-1}$, generalizing the upper bound for $2$-dimensional permutation matrices by Fox~\cite{Fox13}.

For a $d$-dimensional permutation matrix $P$, let $S_P(n)$ be 
the set of $d$-dimensional $n\times \dots \times n$ permutation matrices avoiding $P$. 
The first author~\cite{Cibulka09} proved that for every fixed forbidden matrix $P$, we have
\[
2^{\Omega_d(n)}\cdot (n!)^{d-2} \le |S_P(n)| \le 2^{O_d(n \log\log n)} \cdot (n!)^{d-1 - 1/(d-1)},
\]
where $O_d$ and $\Omega_d$ mean that the constants hidden by the $O$- and $\Omega$-notation depend only on $d$. We use this notation throughout the paper, in particular in Sections~\ref{sec:higherdim} and~\ref{sec_conclusion}.


\subsection{New results}
A $1$-entry in a matrix is identified by the pair $(i,j)$ of the row index $i$
and the column index $j$.  The \emph{distance vector} between the entries
$(i_1, j_1)$ and $(i_2, j_2)$ is $(i_2-i_1, j_2-j_1)$.  
We say that a vector $(d,d')$ is \emph{$r$-repeated} in a permutation matrix $P$ 
if $(d,d')$ occurs as the distance vector of at least $r$ pairs of $1$-entries.
If some vector is $r$-repeated in a permutation matrix $P$, then $P$ has an 
\emph{$r$-repetition}; otherwise, $P$ is \emph{$r$-repetition-free}.

The following theorem shows that the 
F\"{u}redi--Hajnal limit (and hence the Stanley--Wilf limit) is subexponential\footnote{A function $f: \mathbb{R} \rightarrow \mathbb{R}$ \emph{grows exponentially} if 
$f(n) \in 2^{\Theta(n)}$.
Notice that Fox~\cite{Fox13} uses the less restrictive definition where all functions $f(n) \in 2^{n^{\Theta(1)}}$ are exponential.}
for $k$-permutation matrices with no $\Omega(k/\log^6(k))$-repetition.

\begin{theorem}
\label{thm:upper-bound-repetition-free}
Let $k \ge 9$, $r \ge 3$ and let $P$ be an $r$-repetition-free $k$-permutation matrix.
The F\"{u}redi--Hajnal limit of $P$ satisfies
\[
c_P \le 2^{O(r^{1/3} k^{2/3} \log^2 k)}.
\]
\end{theorem}

We say that a $k$-permutation matrix $P$ is \emph{scattered} 
if $P$ is $r$-repetition-free for every $r\ge 4\log_2 k / \log_2\log_2 k$.
In Section~\ref{section_scatter}, we show that as $k$ goes to infinity, 
almost all $k$-permutation matrices are scattered.
This immediately implies the following upper bound on the F\"{u}redi--Hajnal limit
of asymptotically almost all permutation matrices.

\begin{corollary}
\label{cor:random-permutation}
For every $k\ge 9$ and a random $k$-permutation matrix $P$, the F\"{u}redi--Hajnal limit of $P$ satisfies 
\[
c_P \le 2^{O(k^{2/3}\log^{7/3}k / (\log\log k)^{1/3})}
\] 
asymptotically almost surely.
\end{corollary}

We also show upper bounds for some permutation matrices that are far from being scattered.

Let $k$ be a square of an integer and let $G_k$ be the $k\times k$ binary matrix with $1$-entries at positions
$(a+b\sqrt{k}+1, b+a\sqrt{k}+1)$ for every pair $a,b \in \{0, \dots, \sqrt{k}-1\}$.
Fox~\cite{Fox13} used $G_k$ as an example of a permutation matrix for whose 
F\"uredi--Hajnal limit he proved the $2^{\Omega(k^{1/2})}$ lower bound.
We show an upper bound that differs only by a $\log^2(k)$ multiplicative factor in the exponent.

\begin{theorem}
\label{thm:FoxGrid}
For every $k \ge 1$ such that $\sqrt{k} \in \mathbb{N}$, we have
\[
c_{G_k} \le 2^{O(\sqrt{k} \log^2 k)}.
\]
\end{theorem}

In fact, we show a slightly more general upper bound for matrices obtained by so-called \emph{grid products};
see Theorem~\ref{thm:GeneralGrid}.

Let $k \ge 2$ be an even integer. 
Let $X_k$ be the $k \times k$ matrix with $1$-entries on both diagonals; that is, at positions $(i,j)$ where $i,j \in [k]$ and
$i+j = k+1$ or $i-j=0$. 

Given an odd integer $k$, let $\Cross_k$ be the $k$-permutation matrix corresponding
to the permutation $\pi$ satisfying $\pi(i) = i$ for $i$ odd and $\pi(i) = k+1-i$ for $i$ even.
Notice that $\Cross_k$ is contained in $X_{k+1}$.
By a result of the first author~\cite{Cibulka09}, 
the F\"uredi--Hajnal limit of $\Cross_k$ is at least $\Omega(k^2)$.
We show a quasipolynomial upper bound.

\begin{theorem}
\label{thm:CrossMatrix}
Let $k$ be an even integer and let $Q$ be a permutation matrix. 
If $Q$ is contained in $X_k$, then
\[
c_{Q} \le 2^{O(\log^2 k)}.
\]
\end{theorem}

The \emph{density} of a matrix is the ratio of the number of 1-entries 
to the total number of entries of the matrix.
Our general strategy for proving the upper bounds on $c_P$ is first to prove an upper bound on the density of small $P$-avoiding matrices (see Theorem~\ref{thm:fixed-size} and Lemmas~\ref{lem:GeneralGrid} and~\ref{lem:XMatrix}) and then use the following theorem.

\begin{theorem}
\label{thm:framework}
Let $u \in \mathbb{N}$ and $q \in (1/u,1)$.
If a permutation matrix $P$ satisfies
\[
\ex_P(u) < qu^2,
\]
then
\[
c_P \le 2 u^3 u^{\lceil -\log u / \log q \rceil}.
\]
\end{theorem}

Fox~\cite{Fox13} proved that for every $k \in \mathbb{N}$, $\exminor_{J_k}(n) \le 3k 2^{8k} n$. 
The constant in the exponent can be easily decreased from $8$ to $6$.
We further improve it to $4$.

\begin{theorem}
\label{thm:upper-bound-Jk-minor}
Let $k \in \mathbb{N}$. The extremal function for the forbidden $J_k$-minor satisfies
\[
\exminor_{J_k}(n) \le \frac{8}{3} (k+1)^2 \cdot 2^{4k} n.
\]
\end{theorem}

Since every $k$-permutation matrix is contained in $J_k$, we have the following corollary.
\begin{corollary}
For every $k \in \mathbb{N}$ and for every $k$-permutation matrix $P$, the 
F\"{u}redi--Hajnal limit of $P$ satisfies
\[
c_P \le \frac83 (k+1)^2 \cdot 2^{4k}.
\]
\end{corollary}

We also give the following upper bound on the F\"{u}redi--Hajnal limit in terms of the Stanley--Wilf limit, improving the bound $c_P\le O(s_P^{4.5})$ by the first author~\cite{Cibulka09} and a recent unpublished bound $c_P \le s_P^3 \log^{O(1)} s_P$ by Fox~\cite{FoxPersComm}.

\begin{theorem}
\label{thm:swtofh}
For every permutation matrix $P$,
\[
c_P \le O\left(s_P^{2.75} \log s_P\right).
\]
\end{theorem}

We extend the Stanley--Wilf conjecture to higher dimensions, and prove asymptotically matching lower and upper bounds,
improving previous much weaker bounds~\cite{Cibulka09}.

\begin{theorem}
\label{thm:higherdim}
For every $d,k \ge 2$ and every $d$-dimensional $k$-permutation matrix $P$, we have
\[
n^{-O_d(k)} \left(\Omega_d\left(k^{1/(2^{d-1} (d-1))}\right)\right)^n \cdot (n!)^{d-1-1/(d-1)} \le |S_P(n)|
 \le \left(2^{O_d(k)}\right)^n \cdot (n!)^{d-1-1/(d-1)}.
\]
\end{theorem}

We prove Theorem~\ref{thm:framework} in Section~\ref{section_filozofie}, Theorem~\ref{thm:upper-bound-repetition-free} in Section~\ref{section-ref-free-ub}, Theorems~\ref{thm:FoxGrid}~and~\ref{thm:CrossMatrix} in Section~\ref{sec:other-upper-bounds}, Theorem~\ref{thm:upper-bound-Jk-minor} in Section~\ref{section_minor}, Theorem~\ref{thm:swtofh} in Section~\ref{sec:swtofh} and Theorem~\ref{thm:higherdim} in Section~\ref{sec:higherdim}.

All logarithms in this paper are base $2$.


\section{Almost all permutation matrices are scattered}
\label{section_scatter}

\begin{lemma}
\label{lem:repetition-free-onerepeated}
Let $k \in \mathbb{N}$, $r \in [k-1]$ and $d,d' \in \{-k+1, -k+2, \dots, k-1\}$.
The number of $k$-permutation matrices where $(d,d')$ is $r$-repeated is at most $k!/r!$.
\end{lemma}

\begin{proof}
Let $P$ be a permutation matrix where $(d,d')$ is $r$-repeated and let $\pi$ be its 
corresponding permutation.
Using symmetries, we can assume, without loss of generality, that $d,d'>0$.

For every pair $P_{i,j}$, $P_{i+d, j+d'}$ of $1$-entries of $P$ with distance vector 
$(d,d')$, we say that $P_{i,j}$ is a \emph{starting entry} and $P_{i+d, j+d'}$ is 
an \emph{ending entry}. 
Notice that an entry can be both a starting entry and an ending entry.
An \emph{ending row} in $P$ is a row containing an ending entry.

We map $P$ to the pair $(S, \sigma)$ where
\begin{itemize}
\item $S$ is the set of the $r$ ending rows of $P$ and 
\item $\sigma$ is the restriction of $\pi$ on the set of indices of the non-ending rows of $P$.
\end{itemize}
Clearly, there are at most $\binom{k}{r} (k-r)! =  \frac{k!}{r!}$ such pairs.

We now prove that the mapping is injective by showing that if two permutation matrices 
$P$ and $P'$ are mapped to the same pair ($S$, $\sigma$), then $P = P'$.
For contradiction, let $j$ be the leftmost column in which $P$ and $P'$ differ. 

First, consider the case that the $1$-entry in column $j$ in $P$ is in an ending row $i$. 
This implies that $P_{i-d, j-d'} = 1$ and since the column $j-d'$ is to the left of 
the column $j$, we have $P'_{i-d, j-d'}=1$. 
Since $P$ and $P'$ have the same sets of starting rows and $i-d$ is a starting row of $P$,
we have $P'_{i,j}=1$, a contradiction.
By a symmetrical reasoning, we obtain a contradiction in the case when the $1$-entry 
in column $j$ in $P'$ is in an ending row.

In the remaining case the $1$-entries in column $j$ in $P$ and $P'$ are in different
non-ending rows $i_1$ and $i_2$, respectively.
Let $j'$ be the number of $1$-entries in non-ending rows to the left of the $j$th 
column.
Let $i'_1$ and $i'_2$ be the number of non-ending rows above row $i_1$ and $i_2$, 
respectively.
Since $i_1$ and $i_2$ are different non-ending rows, we have $i'_1 \neq i'_2$.
But we also have $\sigma(i'_1+1) = j'+1$ and $\sigma(i'_2+1) = j'+1$, which is a contradiction with the choice of $j$.
\end{proof}

\begin{theorem}
\label{thm:scattered}
Let $k \in \mathbb{N}$ and let $r \in [k]$. 
The number of $k$-permutation matrices with an $r$-repetition is at most
\[
2 k^2 \frac{k!}{r!}.
\]
Consequently, the number of $k$-permutation matrices that are not scattered is in $o(k!)$.
\end{theorem}

\begin{proof}
The distance vector of a pair of $1$-entries in a $k$-permutation matrix can attain
$(2k-2)^2$ different values of the form $(d,d')$, where 
$|d|, |d'| \in \{1, 2, \dots, k-1\}$.
For every $(d,d')$ the vector $(d,d')$ occurs in a matrix the same number of 
times as the vector $(-d,-d')$.
Therefore to get an upper on the number of $k$-permutation matrices 
with an $r$-repetition, it is enough to consider only the $2(k-1)^2$
values of the distance vector where $d, |d'| \in \{1, 2, \dots, k-1\}$.
The first part of Theorem~\ref{thm:scattered} now follows from 
Lemma~\ref{lem:repetition-free-onerepeated}.

The second part of Theorem~\ref{thm:scattered} follows by using the formula to bound the number
of permutation matrices with a $\lceil 4 \log k/ \log\log k\rceil$-repetition.
We have
\begin{align*}
2 k^2 \frac{k!}{\left(4\log k / \log\log k\right)!} 
& \le 2^{1+2\log k} \cdot k! \left(\frac{e \log\log k}{4\log k}\right)^{4\log k/\log\log k} \\
& \le k! \cdot 2^{1+2\log k}\cdot 2^{(\log\log\log k -\log\log k)\cdot 4\log k/\log\log k}  \\
& \le k! \cdot 2^{1-2\log k+o(\log k)} = o(k!).
\qedhere
\end{align*}
\end{proof}


\section{Trade-off between size and density of \texorpdfstring{$P$}{P}-avoiding matrices}
\label{section_filozofie}

In this section we prove Theorem~\ref{thm:framework}.

The \emph{density} of a row of a matrix is the ratio of the number of $1$-entries 
in this row to the number of columns. For a permutation matrix $P$, let $f_P(z,y)$ be the maximum possible number of rows of a binary $P$-avoiding matrix 
with $z$ columns and at least $y$ $1$-entries in every row.
That is, we are interested in matrices where the density of every row is at least 
$\hust = y/z$.

Marcus and Tardos~\cite{MarcusTardos} bounded $f_P(k^2, k)$ for every $k$-permutation matrix $P$, to show that $c_P$ is finite. 
Fox used an upper bound on $f_P(2^{2k}, 2^{k-1})$ to show that $c_P \le 2^{8k}$ for every
$k$-permutation matrix $P$.
In general, Fox's generalization of the Marcus--Tardos recursion~\cite[Lemma 12]{Fox13} requires an upper bound
on $f_P(z,y)$ where $(y-1)^2 < z$, in order to prove a linear upper bound on $\ex_P(n)$.
The next lemma allows us to deduce upper bounds on $c_P$ from bounds on $f_P(z,y)$ 
where $y$ is close to $z$.

By $P^T$ we denote the transpose of $P$. The following proposition is the heart of Theorem~\ref{thm:framework}.

\begin{proposition}
\label{tvrzeni_huste}
Let $P$ be a permutation matrix, $u,h \in \mathbb{N}$ and $\hust \in (0,1)$.
Suppose that for every $z \ge u$, we have 
\begin{enumerate}
\item
\label{cond:framework1}
$f_P(z, \hust z) < h$ \ \ and  
\item
\label{cond:framework2}
$f_{P^T}(z, \hust z) < h.$ 
\end{enumerate}
Then
\[
c_P \le 2 u^3 h^{\lceil -\log u / \log \hust \rceil}.
\]
\end{proposition}

We break the proof of Proposition~\ref{tvrzeni_huste} into a sequence of statements.

Notice that if $\hust \le 1/u$, then condition~\ref{cond:framework1} with $z=u$ implies
that every $h \times u$ matrix with one $1$-entry in every row contains $P$.
This is satisfied only when $P$ is the $1$-permutation matrix.
Then $c_P = 0$ and the conclusion of the proposition is valid.
We therefore further assume that $\hust > 1/u$.

We define a sequence $\hust_i$ of densities of 
$1$-entries as follows. For every $i \ge 1$, let
\[
\hust_i = \max\{1/u, \hust^i\}.
\]
Since $\hust > 1/u$, we have $\hust_1 = \hust$. 
Since $\hust < 1$, 
there is some $i_0 > 1$ such that $\hust_i = \hust^i$ whenever $i < i_0$ 
and $\hust_i = 1/u$ for $i \ge i_0$.
We thus have 
\begin{equation}
\label{eq:hust-frac}
\frac{\hust_i}{\hust_{i-1}} \ge \hust \qquad \text{for every } i \ge 2.
\end{equation}

\begin{lemma}
\label{lemma_sparsifying}
Under the conditions of Proposition~\ref{tvrzeni_huste}, for every $i\ge 1$, we have
\[
f_{P}\left(u^2, \hust_i u^2 \right) < h^i.
\]
\end{lemma}
\begin{proof}
We proceed by induction on $i$. The case $i=1$ follows from condition~\ref{cond:framework1} of Proposition~\ref{tvrzeni_huste}.

Given $i\ge 2$, suppose for contradiction that $A_i$ is an $h^i \times u^2$ binary $P$-avoiding matrix 
with at least $\hust_i u^2 $ $1$-entries in every row.
We split the matrix $A_i$ into $h^{i-1}$ intervals of consecutive $h$-tuples of rows. 
For every $j \in \{1,2,\dots, h^{i-1}\}$, let $A_{i,j}$ be the matrix formed by 
the $j$th interval of rows.

First, assume that for some $j$, the matrix $A_{i,j}$ has at most
$\lfloor\hust_{i-1} u^2\rfloor$ columns with at least one $1$-entry.
Then we consider an $h \times \lfloor\hust_{i-1} u^2\rfloor$ matrix $A'_{i,j}$ created 
from  $A_{i,j}$ by removing some columns with $0$-entries only.
Since $A'_{i,j}$ contains all $1$-entries of $A_{i,j}$, it has at least $\hust_i u^2$ $1$-entries in every row, 
which is at least $\hust \hust_{i-1} u^2$ by \eqref{eq:hust-frac}. 
Condition~\ref{cond:framework1} of Proposition~\ref{tvrzeni_huste} with $z=\lfloor\hust_{i-1} u^2\rfloor$ implies that $A'_{i,j}$ contains $P$.

Now assume that for every $j$, the matrix $A_{i,j}$ has at least $\hust_{i-1} u^2$ columns 
with at least one $1$-entry.
Let $B_i$ be the matrix formed from $A_i$ by contracting the intervals of rows 
forming the matrices $A_{i,j}$. 
Then $B_i$ is an $h^{i-1} \times u^2$ binary matrix with at least $\hust_{i-1} u^2$ $1$-entries in every row.
By the induction hypothesis, $B_i$ contains $P$ and consequently $A_i$ contains $P$.
\end{proof}

\begin{corollary}
\label{cor_tallMatrix}
Under the conditions of Proposition~\ref{tvrzeni_huste}, we have
\[
f_{P}(u^2, u) \le h^{\lceil -\log u / \log \hust \rceil}.
\]
\end{corollary}

\begin{proof}
We use Lemma~\ref{lemma_sparsifying} with $i = \lceil -\log u / \log \hust \rceil$.
Then we have 
\[
\hust^i = \hust^{\lceil-\log u / \log \hust\rceil} \le 1/u,
\]
and so $\hust_i = 1/u$.
By Lemma~\ref{lemma_sparsifying},  
\[
f_{P}(u^2, u) < h^i.
\qedhere
\]
\end{proof}

Let $g_P(z,y)$ be the maximum number of columns of a binary $P$-avoiding matrix 
with $z$ rows and at least $y$ $1$-entries in every column.
Since $g_P(z,y) = f_{P^T}(z,y)$ for every $z,y \in \mathbb{N}$, we have the following corollary.

\begin{corollary}
\label{cor_wideMatrix}
Under the conditions of Proposition~\ref{tvrzeni_huste}, we have
\[
g_{P}(u^2, u) \le h^{\lceil -\log u / \log \hust \rceil}.
\qedhere
\]
\qed
\end{corollary}

\begin{proof}[Proof of Proposition~\ref{tvrzeni_huste}]
By Fox's generalized Marcus--Tardos recursion~\cite[Lemma 12]{Fox13}, 
for every permutation matrix $P$ and for all positive integers $n,s,t$ with 
$s \le t$, we have
\[
\ex_P(tn) \le \ex_P(s-1)\cdot \ex_P(n) + \ex_P(t)\cdot n \cdot(f_P(t,s) + g_P(t,s)).
\]
Marcus and Tardos~\cite{MarcusTardos} used the recursion with parameters $t=k^2$ and $s=k$. 
We choose the parameters $t=u^2$ and $s=u$.

Let 
\[
\hloubka = h^{\lceil -\log u / \log \hust \rceil}.
\]
By Corollaries~\ref{cor_tallMatrix}~and~\ref{cor_wideMatrix} and
by the trivial estimates $\ex_P(s-1) \le (s-1)^2$ and $\ex_P(t) \le t^2$,
we have
\begin{equation}
\label{eq:fox}
\ex_P(u^2 n) \le (u-1)^2 \cdot \ex_P(n) + u^4 n \cdot 2 \hloubka.
\end{equation}

Arratia~\cite{Arratia99} proved that $|S_P(n)|$ is supermultiplicative in $n$ 
for every fixed permutation matrix $P$. 
An analogous proof shows that for every permutation matrix $P$, 
the extremal function $\ex_P(n)$ is superadditive~\cite[Lemma 1(ii)]{PachTardos}; 
that is, for every $m, n \in \mathbb{N}$, we have
\[
\ex_P(m+n) \ge \ex_P(m) + \ex_P(n).
\]
Consequently, for every permutation matrix $P$ and $n, \alpha \in \mathbb{N}$, we have
\begin{equation}
\label{eq:superadditivity}
\ex_P(\alpha n) \ge \alpha \cdot \ex_P(n).
\end{equation}

By combining inequality~\eqref{eq:superadditivity} for $\alpha = u^2$ with inequality~\eqref{eq:fox}, we obtain
\begin{align*}
u^2 \cdot \ex_P(n) 
& \le \ex_P(u^2 n) 
\le (u-1)^2 \cdot \ex_P(n) + u^4 n \cdot 2 \hloubka \\
(2u-1) \cdot \ex_P(n) & \le u^4 n \cdot 2 \hloubka \\
\ex_P(n) & \le 2 u^3 h^{\lceil -\log u / \log \hust \rceil} n.
\qedhere 
\end{align*}
\end{proof}

\begin{proof}[Proof of Theorem~\ref{thm:framework}]
For every $z \ge u$, if a $u \times z$ matrix $A$ contains at least $qz$ $1$-entries in every
row, then $A$ contains at least $quz$ $1$-entries.
Thus, we can select $u$ columns having together at least $qu^2$ $1$-entries.
Consequently, the condition $\ex_P(u) < qu^2$ implies condition~\ref{cond:framework1} of Proposition~\ref{tvrzeni_huste} with $h=u$.
The validity of condition~\ref{cond:framework2} of Proposition~\ref{tvrzeni_huste} follows from the fact that $\ex_P(u)=\ex_{P^T}(u)$.
\end{proof}


\section{Repetition-free permutation matrices}
\label{section-ref-free-ub}

In this section we prove Theorem~\ref{thm:upper-bound-repetition-free}.
We first show that for given $k$ and $r$ and an $r$-repetition-free permutation matrix $P$,
every $3k \times 3k$ matrix with a sufficiently small number of $0$-entries in every row 
and every column contains $P$ (see Lemma~\ref{lem:repetition-free-in-3k}).
We then show that every $4k \times 4k$ matrix with a sufficiently small total number of 
$0$-entries contains $P$ (see Theorem~\ref{thm:fixed-size}). 
Theorem~\ref{thm:upper-bound-repetition-free} then follows by Theorem~\ref{thm:framework}.

We analyze a straightforward greedy algorithm for finding an occurrence of a $k$-permutation
matrix $P$ on a given $k$-tuple of rows of a binary matrix $B$.
In this setting, every $1$-entry of $P$ has a prescribed row of $B$ in which it is 
to be mapped. 
For every $j$, let $\row_{j}$ be the row of $B$ in which 
the $1$-entry from the $j$th column of $P$ is to be mapped.
Figure~\ref{fig:algo} shows an example of the execution of the algorithm.

\begin{figure}
\begin{center}
{
  \ifpdf\includegraphics{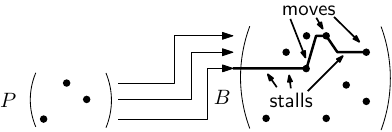}\fi
}
\end{center}
\caption{\label{fig:algo} An example of the execution of the algorithm for finding an occurrence of 
a $k$-permutation matrix $P$ on a fixed $k$-tuple of rows of a binary matrix $B$.}
\end{figure}

In every step of the algorithm, one entry of $B$ is inspected.
The entry inspected in the $i$th step of the algorithm is always in the $i$th
column of $B$.
The entry of $B$ inspected in the first step lies in the row $r_1$.
In every step, the algorithm does the following.
If the inspected entry is $0$, the algorithm stays in the same
row for the next step, and we say that it \emph{stalls}.
If the inspected entry is $1$ and the current row is $\row_{j}$
for some $j \le k-1$, the algorithm goes to the row $\row_{j+1}$
for the next step, and we say that the algorithm \emph{moves}.
If the inspected entry is $1$ and the current row is $\row_{k}$,
then an occurrence of $P$ has been found and the algorithm terminates.

We note that if the algorithm fails to find an occurrence of $P$, then the given $k$-tuple 
of rows does not contain an occurrence of $P$. This fact is, however, not used in the proof.

Let $B$ be a $3k\times 3k$ matrix.
We simultaneously run $2k+1$ instances of the algorithm, 
one for every $k$-tuple of consecutive rows of $B$. 
If an instance of the algorithm does not find an occurrence of $P$, then 
at most $k-1$ of its steps are moves. 
Hence, if at least one of the instances makes at least $k$ moves, $B$ contains $P$.

Given integers $k \ge 9$ and $r \ge 3$, let
\begin{align*}
w &= \left\lceil \frac{35}{24} \left(\frac{k}{r}\right)^{1/3}\right\rceil \qquad \text{ and } \\
v &= \frac13 \left(\frac{k}{r}\right)^{1/3}.
\end{align*}
For every $k \ge 9$ and $r \ge 3$, we have
\begin{equation}
\label{eq:w-vs-k-third}
w \le \frac{k}{3}.
\end{equation}
Indeed, if $k \le 11$, then $w \le 3 \le k/3$. 
Otherwise, $w < \lceil 1.1 k^{1/3}\rceil < 1.1 k^{1/3} + 1 < k/3$.

The following claim is the main part of the proof.

\begin{lemma}
\label{lem:moves-now-or-later}
Let $k \ge 9$ and let $B$ be a $3k\times 3k$ binary matrix with
at most $v$ $0$-entries in every row and in every 
column.
Let $r \ge 3$ and let $P$ be an $r$-repetition-free $k$-permutation matrix.
For every $j \in \{1, 2, \dots, 3k-w\}$, either at least $3k/4$ instances of the algorithm
make a move in the $j$th step or 
the sum of the numbers of moves made by the instances in steps $j,\allowbreak 
j+1, \dots,\allowbreak j+w-1$ is at least $3kw/4$.
\end{lemma}

\begin{proof}
If an instance of the algorithm stalls (moves) after inspecting $B_{i,j}$, then we
say that the instance \emph{stalls (moves) on $B_{i,j}$}.

Assume that at most $3k/4$ of the instances make a move in the $j$th step.
Consider the instances stalled on some $B_{i,j}=0$.
Since the $i$th row of $B$ contains at most $v$ $0$-entries, all the 
instances stalled on $B_{i,j}$ will move simultaneously in $j'$th step for some 
$j' \in \{j+1, j+2, \dots, j+\lfloor v\rfloor\}$.

For every $l \ge 1$, let $M_l$ be the set of instances that stall on $B_{i,j}$
and make a move in each of the steps $j', j'+1, \dots, j'+l-1$.
Thus $M_1$ is the set of all instances stalled on $B_{i,j}$.
For every $l \ge 1$, the set $M_{l} \setminus M_{l+1}$ is the set of instances
from $M_l$ that are stalled on a $0$-entry in the $(j'+l)$th column of $B$.

We now use the fact that $P$ is $r$-repetition-free to bound the size of 
$M_{l} \setminus M_{l+1}$.
By the selection of the $k$-tuples of rows on which the instances are running,
every instance in $M_l$ made a different number of moves before the $j$th step.
Consider a $0$-entry $B_{i', j'+l}$.
There are at least as many occurrences of the distance vector $(i'-i, l)$
between two $1$-entries of $P$ as there are instances from $M_{l}$
stalled on $B_{i', j'+l}$.
Thus, on each of the $0$-entries in the $(j'+l)$th column of $B$, at most 
$r$ of the instances from $M_l$ are stalled.
Since every column of $B$ contains at most $v$ $0$-entries, we have
\begin{align*}
|M_l \setminus M_{l+1}| &\le v r \qquad \text{ and consequently } \\
|M_1 \setminus M_{w-\lfloor v\rfloor}| &\le (w-v) v r.
\end{align*}
All instances in $M_{w-\lfloor v\rfloor}$ were stalled on $B_{i,j}$ and made at least 
$w- v$ moves in steps $j,\allowbreak j+1, \dots,\allowbreak j+w-1$.

The $j$th column of $B$ contains at most $v$ $0$-entries.
Each of the more than $5k/4$ instances stalled in the $j$th step is stalled on one of 
these $0$-entries.
We thus conclude that the number of instances that made at least $w-v$ moves 
in steps $j,\allowbreak j+1, \dots,\allowbreak j+w-1$ is at least
\begin{align*}
\frac54 k - v (w-v) v r 
&\ge \frac54 k - \frac19 \left(\frac{k}{r}\right)^{2/3} r 
\left(\frac{35}{24}-\frac13\right)\left(\frac{k}{r}\right)^{1/3} \\
 &=  \frac54 k - \frac19 \cdot \frac98 k \\
 &= \frac98 k.
\end{align*}

We have
\[
\frac{w-v}{w} \ge 1- \frac{8}{35} = \frac{27}{35}
\]
and thus the number of moves in steps $j,\allowbreak j+1, \dots,\allowbreak j+w-1$ is at least 
\[
\frac98 k (w-v) \ge \frac98 \cdot \frac{27}{35} kw > \frac34 kw.
\]
\vskip-18pt
\qedhere
\end{proof}

A \emph{tight occurrence} of a $k \times k$ binary matrix $P$ in a matrix $B$ is
an occurrence of $P$ on some $k$ consecutive rows of $B$.

\begin{lemma}
\label{lem:repetition-free-in-3k}
Let $k \ge 9$ and $r\ge 3$. Let $B$ be a $3k\times 3k$ binary matrix and let
$P$ be an $r$-repetition-free $k$-permutation matrix.
If $B$ has at most $(1/3)(k/r)^{1/3}$ $0$-entries in every row and in every 
column then $B$ contains $P$.
Moreover, the occurrence of $P$ in $B$ is tight.
\end{lemma}

\begin{proof}
We assign types to some of the steps of the algorithm.
The assignment starts with the step $1$.
Let $j$ be a step considered during the assignment procedure.
If $j > 3k-w$, we finish the assignment procedure.
If at least $3k/4$ instances make a move in the $j$th step,
then we say that the $j$-th step is of \emph{type $1$} and proceed to 
the $(j+1)$st step, otherwise the steps $j,\allowbreak j+1, \dots,\allowbreak j+w-1$ are of 
\emph{type $2$} and we proceed to the $(j+w)$th step.
The number of moves in every step of type $1$ is at least $3k/4$.
By Lemma~\ref{lem:moves-now-or-later}, the 
average number of moves in steps of type $2$ is at least $3k/4$.
The total number of moves is thus at least
\[
(3k-w)\cdot\frac34 k \stackrel{\text{by~\eqref{eq:w-vs-k-third}}}{\ge} \left(3-\frac13\right) \cdot k \cdot \frac34 k = 2 k^2 > (2k+1)(k-1)
\]
and so at least one of the instances made $k$ moves and found an occurrence of $P$.
\end{proof}

\begin{theorem}
\label{thm:fixed-size}
Let $k \ge 9$ and $r \ge 3$.
Let $A$ be a $4k\times 4k$ binary matrix and let $P$ be an $r$-repetition-free 
$k$-permutation matrix.
If $A$ has at most $(k/3) (k/r)^{1/3}$ $0$-entries then $A$ contains $P$.
\end{theorem}

\begin{proof}[Proof of Theorem~\ref{thm:fixed-size}]
The matrix $A$ has at most $k$ rows with more than $v$ $0$-entries
and at most $k$ columns with more than $v$ $0$-entries.
Thus, after removing $k$ rows and $k$ columns with the largest number of $0$-entries,
we obtain a matrix $B$ satisfying the requirements of Lemma~\ref{lem:repetition-free-in-3k},
thus containing $P$.
\end{proof}

\begin{proof}[Proof of Theorem~\ref{thm:upper-bound-repetition-free}]
Let $k \ge 9$, $r \ge 3$ and let $P$ be an $r$-repetition-free $k$-permutation matrix.
By Theorem~\ref{thm:fixed-size}, we can use Theorem~\ref{thm:framework} with
\begin{align*}
u &= 4k \qquad \text{ and } \\
q &= 1 - \frac{1}{48 r^{1/3} k^{2/3}}.
\end{align*}
We have 
\begin{align*}
\log u &\in O(\log k) \qquad \text{ and } \\
\log q &\le -\frac{\log e}{48 r^{1/3} k^{2/3}} \le -\frac{1}{34 r^{1/3} k^{2/3}}.
\end{align*}
By Theorem~\ref{thm:framework},
\[
c_P \le 2u^{3+\lceil -\log u / \log q \rceil} 
\le 2 u^{4 + 34 r^{1/3} k^{2/3} \log u } 
= 2^{1+ 4 \log u + 34 r^{1/3} k^{2/3} \log^2 u}.
\]

Thus 
\[
c_P \le 2^{O(r^{1/3} k^{2/3}\log^2 k)}.
\qedhere
\]
\end{proof}


\section{Some additional upper bounds}
\label{sec:other-upper-bounds}
In this section, we show subexponential upper bounds on $c_P$ for a few special 
matrices that are far from being scattered.

\subsection{Grid products}
Consider a $k$-permutation matrix $P$ with $1$-entries at positions $(i,\pi(i))$ for every $i \in [k]$
and an $l$-permutation matrix $Q$ with $1$-entries at positions $(j,\rho(j))$ for every $j \in [l]$.
We define the \emph{grid product} $R = P \# Q$ to be the $(kl)$-permutation matrix with $1$-entries
at positions $((j-1)\cdot k + i, (\pi(i)-1) \cdot l + \rho(j))$ for every $i \in [k]$
and $j \in [l]$. See Figure~\ref{fig:gridI3I4} for an example.

\begin{figure}
\begin{center}
{
  \ifpdf\includegraphics{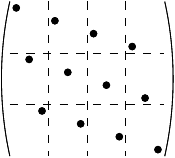}\fi
}
\end{center}
\caption{\label{fig:gridI3I4} The grid product $I_3 \# I_4$.}
\end{figure}

\begin{lemma}
\label{lem:GeneralGrid}
Let $k,l \ge 2$ and $m \ge 1$. Let $P$ be a $k$-permutation matrix, $Q$ an $l$-permutation matrix, $t = kl$ and $R = P \# Q$.
We have
\[
\ex_{R}(mt) < (mt)^2 - k\cdot((ml-1)^2-\ex_Q(ml-1)).
\]
\end{lemma}
\begin{proof}
Let $z = (ml-1)^2-\ex_Q(ml-1)$, that is, the minimum number of $0$-entries in an 
$(ml-1) \times (ml-1)$ $Q$-avoiding matrix.
Let $A$ be an $mt \times mt$ matrix with at most $z k - 1$ $0$-entries. 
We show that $A$ contains $R$.

We cut $A$ into $k$ rectangles of width $ml$, rearrange them on top of each other with a small vertical displacement determined by $P$, and form their ``superposition'' matrix $A'$.
See Figure~\ref{fig:gridProof}.
Formally, let $A'$ be the $(mt-k) \times m l$ matrix
such that for every $i \in [mt-k]$ and $j \in [m l]$, ${A'}_{ij}=1$ if and only if 
for every $\alpha \in [k]$, $a_{i+\alpha-1, j+ml(\pi(\alpha)-1)} = 1$.
Since every element of $A$ is used to define at most one element of $A'$, the number of $0$-entries in $A'$ is at most $(zk - 1)$.
Notice that if $A'$ contains the matrix $Q'$ obtained from $Q$ by inserting 
$k-1$ rows full of zeros between every pair of consecutive rows of $Q$, then $A$ contains $R$.

\begin{figure}
\begin{center}
{
  \ifpdf\includegraphics{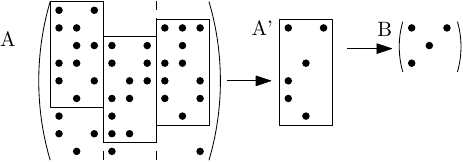}\fi
}
\end{center}
\caption{\label{fig:gridProof} Construction of the matrices $A'$ and $B$ in the proof of Lemma~\ref{lem:GeneralGrid}.
In this example, we have $k=l=3$, $m=1$, and $\pi(1)=1$, $\pi(2)=3$, $\pi(3)=2$.}
\end{figure}

Let $B$ be the $(ml-1) \times m l$ matrix formed by the set of rows 
$\{p+\alpha k: \alpha \in \{0, \dots, ml-2\}\}$ of $A'$ where $p$ is chosen 
from $[k]$ so as to minimize the number of $0$-entries of $B$.
Thus $B$ has at most $z-1$ $0$-entries and so it contains $Q$.
An occurrence of $Q$ in $B$ implies an occurrence of $Q'$ in $A'$.
Consequently, $A$ contains $R$.
\end{proof}

\begin{theorem}
\label{thm:GeneralGrid}
Let $k,l \ge 2$.
Let $Q$ be an $l$-permutation matrix with F\"{u}redi--Hajnal constant $c_Q \ge 3$ 
and let $P$ be a $k$-permutation matrix.
Let $R$ be the $kl$-permutation matrix $P \# Q$. 
Then
\[
c_R \le 2^{O(k \log^2 (c_Q k))}.
\]
\end{theorem}
\begin{proof}
We use Theorem~\ref{thm:framework} with $u = mkl$, where $m = \lceil 2 c_Q / l \rceil$.
Since $c_Q \ge 2(l-1)$ (see e.g.~\cite[Claim 1]{Cibulka09}) and $l \ge 2$, 
we have $m < (2c_Q+l)/l \le 3 c_Q / l$ and $u \le 3 c_Q k$.

Since $\ex_{Q}(ml-1) \le c_Q (ml-1)$, Lemma~\ref{lem:GeneralGrid} implies
\begin{align*}
\ex_{R}(u) 
&< u^2 - k\cdot ((ml-1)^2 - c_Q (ml-1)) \\
&= u^2 - k\cdot(ml-1)(ml-1- c_Q) \\
&\le u^2 - k\cdot(2c_Q-1)(c_Q-1) \qquad \text{ since $m \ge 2 c_Q / l$} \\
&\le u^2 - kc^2_Q \qquad \text{ since $c_Q \ge 3$}\\
&\le u^2 - km^2l^2 / 9 \qquad \text{ since $m \le 3 c_Q / l$} \\
&= u^2 (1 - 1 / (9 k)).
\end{align*}
That is, $\ex_{R}(u) \le u^2 q$, where $q = 1-1/(9k)$.
We estimate
\[
\log q = \log\big(1-1/(9k)\big) < -\log(e)/(9k).
\]
By Theorem~\ref{thm:framework}, we have
\[
c_R \le 2 u^3 u^{\lceil - \log u / \log q \rceil}
\le u^{O(k \log u)}
\le 2^{O(k \log^2 (c_Q k))}. \qedhere
\]
\end{proof}

\begin{proof}[Proof of Theorem~\ref{thm:FoxGrid}]
We have $G_k = I_{\sqrt{k}} \# I_{\sqrt{k}}$, where $I_{\sqrt{k}}$ is the $\sqrt{k} \times \sqrt{k}$ 
identity matrix. It is known that $c_{I_{\sqrt{k}}}=2(\sqrt{k}-1)$ (see e.g.~\cite[Claim 1]{Cibulka09}).
Thus, using Theorem~\ref{thm:GeneralGrid} with $P=Q=I_{\sqrt{k}}$, we get
\[
c_{G_k} 
\le 2^{O(\sqrt{k} \log^2 (2 \sqrt{k} \sqrt{k}))}
\le 2^{O(\sqrt{k} \log^2 k)}. \qedhere
\]
\end{proof}

\begin{remark}
Guillemot and Marx~\cite{GuillemotMarx} define the \emph{canonical $r \times s$ grid permutation}
as $I_r \# J_s$, where $J_s$ is the \emph{reversal matrix}
with $1$-entries at positions $(i,j)$ satisfying $i+j = s+1$.
Theorem~\ref{thm:GeneralGrid} thus gives the upper bound $c_{I_r \# J_s} \le 2^{O(r \log^2(rs))}$.
In general, the same asymptotic upper bound is obtained for any grid product $P \# Q$ where 
$P$ is an $r$-permutation matrix and $Q$ is an $s$-permutation matrix with $c_Q$ polynomial in $s$.
\end{remark}

\subsection{The cross matrix}

\begin{lemma}
\label{lem:XMatrix}
For every integer $k \ge 6$ that is a multiple of $6$, we have
\[
\ex_{X_k}(2k) < (2k)^2 - k^2/18.
\]
\end{lemma}
\begin{proof}
Let $A$ be a $2k \times 2k$ matrix with at most $k^2/18$ $0$-entries. 

Given $d \in \{-2k+1, -2k+2, \dots, 2k-2, 2k-1\}$, the \emph{$d$-diagonal} of $A$ is the set of entries
at positions $(i,j)$ satisfying $i,j \in [2k]$ and $i-j = d$.
Given $c \in \{2, 3, \dots, 4k\}$, the \emph{$c$-antidiagonal} of $A$ is the set of entries
at positions $(i,j)$ satisfying $i,j \in [2k]$ and $i+j = c$.

Since $A$ has at most $k^2/18$ $0$-entries, there exists $d \in \{-k/6, -k/6+1, \dots, k/6\}$
such that the $d$-diagonal contains at most $k/6$ $0$-entries.
Analogously, there is $c \in \{2k-k/6+1, 2k-k/6+2, \dots, 2k+k/6+1\}$
such that the $c$-antidiagonal contains at most $k/6$ $0$-entries.

Let $r = (c+d)/2$ and $s = (c-d)/2$.
Note that if $c$ and $d$ have the same parity, then the entry at position $(r,s)$ is the intersection of
the $d$-diagonal and the $c$-antidiagonal.
We have
\begin{equation}
\label{eq:XMatrixRange}
k - \frac{k}{6} + \frac12 \le r,s \le k + \frac{k}{6} + \frac12.
\end{equation}
If $c$ and $d$ have the same parity, for every $i \in \{1, \dots, 5k/6\}$, let $S_i$ be the set of entries at positions $(r-i, s-i)$, $(r-i, s+i)$, $(r+i, s-i)$ and $(r+i, s+i)$. Similarly, if $c$ and $d$ have the opposite parity, for every $i \in \{1/2, 3/2, \dots, 5k/6-1/2\}$, let $S_i$ be the set of entries at positions $(r-i, s-i)$, $(r-i, s+i)$, $(r+i, s-i)$ and $(r+i, s+i)$. 
Note that by~\eqref{eq:XMatrixRange}, the entries of each such $S_i$ lie in $A$.
Additionally, all these entries lie in the union of the $d$-diagonal and the $c$-antidiagonal, thus there are only
at most $2k/6$ $0$-entries among them.

Let $I$ be the set of at least $k/2$ indices $i$ such that all the four entries of $S_i$ are $1$-entries.
The set $\bigcup_{i\in I} S_i$ forms an occurrence of $X_k$ in $A$.
\end{proof}

\begin{proof}[Proof of Theorem~\ref{thm:CrossMatrix}]
We use Theorem~\ref{thm:framework} with $u = 2k$ and $q=1-1/72$. In particular, $\log q$ is a negative constant.
By Lemma~\ref{lem:XMatrix}, $\ex_{X_k}(u) < qu^2$ and thus by Theorem~\ref{thm:framework}, we have
\[
c_{Q} \le 2 u^3 u^{\lceil -\log u / \log q \rceil}
\le 2^{O(\log^2 k)}.
\qedhere
\]
for every permutation matrix $Q$ contained in $X_k$.
\end{proof}


\section{General permutation matrices}
\label{section_minor}

In this section we prove Theorem~\ref{thm:upper-bound-Jk-minor}.

Given $r,k,s,t \in \mathbb{N}$, 
let $f_{r,k}(t,s)$ be the maximum number of rows in a $J_{r,k}$-minor-free 
binary matrix with $t$ columns where each row contains at least $s$ $1$-entries.
Notice that if $s > t$ then $f_{r,k}(t,s) = 0$ since a matrix cannot have more $1$-entries 
in a row than the number of columns.

Fox~\cite[Lemma 14]{Fox13} proved the following recurrence for every 
$r,k,s,t \in \mathbb{N}$ with $t$ and $s$ even and satisfying $s \le t$:
\begin{equation}
\label{eq_rekurence_frk}
f_{r,k}(t,s) \le 2 f_{r,k}(t/2,s) + 2 f_{r,k-1}(t/2, s/2).
\end{equation}
Then he used it to prove the following upper bound on $f_{r,k}(t,s)$~\cite[Lemma 15]{Fox13}:
\begin{equation}
\label{eq_upper_frk}
f_{r,k}(t,s) \le r 2^{k-1} t^2/s.
\end{equation}
In Lemma~\ref{lemma_frk} we further improve this upper bound by a factor of $2^{k-1}/s$. Fox~\cite{Fox13} used the upper bound~\eqref{eq_upper_frk} only with $s=2^{k-1}$. We use our Lemma~\ref{lemma_frk} also with $s$ equal to larger powers of $2$, which better approximate the number of $1$-entries in a given row.

\begin{lemma}
\label{lemma_frk}
For every $r,k,s,t \in \mathbb{N}$ where $t$ and $s$ are powers of $2$ and 
$t \ge s \ge 2^{k-1}$, we have
\[
f_{r,k}(t,s) \le r 2^{2k-2} (t/s)^2.
\]
\end{lemma}

\begin{proof}
The claim is trivially true when $s>t$, because then
\[
f_{r,k}(t,s) = 0 \le r 2^{2k-2} (t/s)^2.
\]
The claim is also true when $k=1$ and $t \ge s \ge 1$:
\[
f_{r,1}(t,s) = r-1 \le r 2^{2k-2} (t/s)^2.
\]
We proceed by induction on $k + \log(t/s)$, which is an integer since $s$ and 
$t$ are powers of $2$. 
By~\eqref{eq_rekurence_frk}, we have
\begin{align*}
f_{r,k}(t,s) 
& \le 2 f_{r,k}(t/2,s) + 2 f_{r,k-1}(t/2, s/2) \\
& \le 2 r 2^{2k-2} (t/2s)^2 + 2 r 2^{2k-4} (t/s)^2 \\
& \le r 2^{2k-2} (t/s)^2.
\qedhere
\end{align*}

\end{proof}

\begin{proof}[Proof of Theorem~\ref{thm:upper-bound-Jk-minor}]
Fix $k$ and let $t = 2^{2k}$. Further, let $s_i = 2^{i}$ for every $i \in \{k-1, k, \dots, 2k\}$.
Let $A$ be an $n \times n$ binary matrix.
We discard $n \ \text{mod}\ t$ rightmost columns and bottommost rows of $A$ and
split the rest of  $A$ into $\lfloor n/t \rfloor \times \lfloor n/t \rfloor$ blocks 
of size $t \times t$.

We say that a $t \times t$ block of $A$ is \emph{$s_i$-wide}
if it has at least $s_i$ nonempty columns, and \emph{$s_i$-tall} if it has at least $s_i$ nonempty columns.

By contracting the blocks, we form an $\lfloor n/t \rfloor \times \lfloor n/t \rfloor$ matrix that does not contain $J_k$ as an interval minor, and thus it has at most $\exminor_{J_k}(\lfloor n/t \rfloor)$ $1$-entries.
The number of $1$-entries in the blocks that are neither $s_{k-1}$-wide nor
$s_{k-1}$-tall is thus at most 
\[
\exminor_{J_k}(s_{k-1}) \cdot \exminor_{J_k}(\lfloor n/t \rfloor).
\]

If $A$ has $\lfloor n/t \rfloor f_{k,k}(t,s_i)$ $s_i$-wide blocks then some 
$f_{k,k}(t,s_i)$ of them are on the same columns of $A$. This implies that $J_k$ is an interval minor of $A$.
An $s_i$-wide block that is neither $s_{i+1}$-wide nor $s_{i+1}$-tall 
contains at most $\exminor_{J_k}(s_{i+1})$ $1$-entries.
The total number of $1$-entries in blocks that are $s_i$-wide but neither 
$s_{i+1}$-wide nor $s_{i+1}$-tall in an $n \times n$ $J_k$-avoiding matrix
is thus at most
\[
\exminor_{J_k}(s_{i+1}) \cdot \frac{n}{t} \cdot f_{k,k}(t, s_i).
\]
The same bound holds for blocks that are $s_i$-high but neither $s_{i+1}$-wide nor 
$s_{i+1}$-high.
The number of entries in the discarded rows and columns is together smaller than 
$2t n$.

The claim of Theorem~\ref{thm:upper-bound-Jk-minor} is clearly true whenever $n \le 2^{2k}$. 
By induction on $n$ and by Lemma~\ref{lemma_frk}, we have
\begin{align*}
\exminor_{J_k}(n) 
& \le \exminor_{J_k}(s_{k-1}) \cdot \exminor_{J_k}(\lfloor n/t \rfloor)+ 2 t n + 
\sum_{i=k-1}^{2k-1} 2 \cdot \exminor_{J_k}(s_{i+1}) \cdot \frac{n}{t} \cdot f_{k,k}(t, s_i) \\
& \le 2^{2(k-1)} \cdot \frac{8}{3} (k+1)^2 \cdot 2^{4k}\cdot \frac{n}{2^{2k}} + 
2^{2k+1} n + 
2 \cdot \frac{n}{2^{2k}} \cdot \sum_{i=k-1}^{2k-1} (s_{i+1})^2 \cdot k \cdot 2^{2k-2}\cdot \frac{t^2}{s_i^2} \\
& \le \frac{2}{3} (k+1)^2 \cdot 2^{4k} n + 
2^{2k+1} n + 
2n \cdot 2^{-2k} \cdot 2^{2k-2} k \cdot \sum_{i=k-1}^{2k-1} 2^{2i+2} \cdot 2^{4k-2i}\\
& \le \frac{2}{3} (k+1)^2 \cdot 2^{4k} n + 2^{2k+1} n + 2k (k+1) \cdot 2^{4k}n \\
& \le \frac{8}{3} (k+1)^2 \cdot 2^{4k} n.
\qedhere
\end{align*}

\end{proof}


\section{Improved upper bound on the F\"{u}redi--Hajnal limit in terms of the Stanley--Wilf limit}
\label{sec:swtofh}

In this section we prove Theorem~\ref{thm:swtofh}. To achieve this, we refine the recursive method used by Marcus and Tardos~\cite{MarcusTardos} to obtain a finite upper bound on $c_P$, and extended by the first author~\cite{Cibulka09} to obtain an upper bound on $c_P$ in terms of $s_P$. Our refinement combines two additional ideas: distinguishing the density of blocks by powers of $2$, as in the proof of Theorem~\ref{thm:upper-bound-Jk-minor},
and performing an induction with general rectangular $m\times n$ matrices instead of square matrices. 

Throughout this section, $P$ is a fixed $k\times k$ permutation matrix.

\subsection{Height and width compression}

The inductive step in Marcus--Tardos's proof of the F\"{u}redi--Hajnal conjecture~\cite{MarcusTardos} involves 
contracting blocks of size $k^2 \times k^2$ in an $n\times n$ matrix into single entries, resulting in a matrix of size $n/k^2 \times n/k^2$.
This operation can be regarded as a composition of two operations,
\emph{height compression} and \emph{width compression}, which we define next.

Let $m,n$ and $t$ be positive integers, with $m$ divisible by $t$, and let $A$ be a binary $m\times n$ matrix.
\emph{Height $t$-compression} (of $A$) is an operation that consists of splitting $A$ into ``vertical'' blocks of size $t\times 1$, and replacing every such block $B$ with a $1$-entry if $B$ contains at least one $1$-entry, and with a $0$-entry otherwise. As a result of applying height $t$-compression to $A$ we obtain a matrix of size $(m/t) \times n$.
\emph{Width $t$-compression} differs only by interchanging the role of rows and columns: we assume $n$ divisible by $t$, we are splitting $A$ into ``horizontal'' blocks of size $1\times t$, replacing them analogously with single entries, and the result is a matrix of size $m \times (n/t)$.
See Figure~\ref{fig:compressions}.

\begin{figure}
\begin{center}
{
  \ifpdf\includegraphics{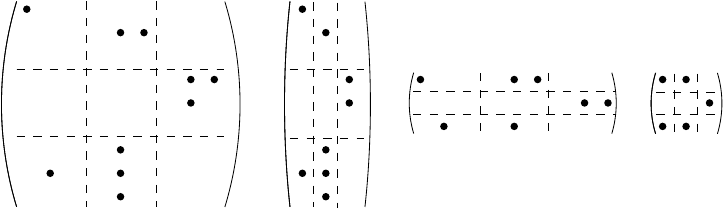}\fi
}
\end{center}
\caption{\label{fig:compressions} 
An example of a matrix subdivided into blocks, the results of its width and height $3$-compression, and the result of the block contraction.}
\end{figure}

We use the following notation from Section~\ref{section_filozofie}:
we let $f_P(t,s)$ be the maximum possible number of rows of a binary $P$-avoiding matrix 
with $t$ columns and at least $s$ $1$-entries in every row.
Then $f_{P^T}(t,s)$ is the maximum possible number of columns of a binary $P$-avoiding matrix 
with $t$ rows and at least $s$ $1$-entries in every column.
We also define $g_P(t,s)$ as the maximum possible number of $1$-entries in a binary $P$-avoiding
matrix with $t$ columns and at least $s$ $1$-entries in every row.

The next lemma gives a recursive upper bound on the number of $1$-entries in a binary $P$-avoiding matrix, obtained by height compression and width compression.

\begin{lemma}
\label{lem:compression}
Let $m,n,t,s$ be positive integers satisfying $t \ge s \ge 1$.
Then for $n$ divisible by $t$, we have
\[
\ex_P(m, n) \le \frac{n}{t} \cdot g_P(t,s) + (s-1)\cdot\ex_P\left(m,\frac{n}{t}\right),
\]
and for $m$ divisible by $t$, we have
\[
\ex_P(m, n) \le \frac{m}{t} \cdot g_{P^T}(t,s) + (s-1)\cdot\ex_P\left(\frac{m}{t}, n\right).
\]
\end{lemma}

\begin{proof}
Let $A$ be an $m \times n$ binary $P$-avoiding matrix.
For every $i \in [n/t]$, let $S_i$ be the $m\times t$ block of $A$ formed by the columns $(i-1) \cdot t+1, (i-1) \cdot t+2, \dots, i \cdot t$. 
We call a row of $S_i$ \emph{wide} if it has at least $s$ $1$-entries.
The wide rows of $S_i$ contain together at most $g_P(t,s)$ $1$-entries.
From each $S_i$ we remove all $1$-entries in the wide rows, and apply width $t$-compression to the resulting $m \times n$ matrix. This operation transforms each $S_i$ into a single column, and as a result we obtain an $m \times (n/t)$ binary $P$-avoiding matrix.
Since each contracted $1\times t$ block contained at most $s-1$ $1$-entries, the number of $1$-entries in $A$ was at most $(n/t) \cdot g_P(t,s) + (s-1)\cdot\ex_P(m, n/t)$.

The second inequality is obtained similarly using height $t$-compression.
\end{proof}

\subsection{Marcus--Tardos's recursion as a composition of height and width compression}

To illustrate our method of proving Theorem~\ref{thm:swtofh} in a simpler setting,
we first express the Marcus--Tardos's proof using height and width compression.

Marcus and Tardos~\cite{MarcusTardos} observed that $f_P(k^2,k) \le k \cdot \binom{k^2}{k}$. 
Since each row in a matrix with $k^2$ columns has at most $k^2$ $1$-entries,
we have $g_P(k^2,k) \le k^3 \cdot \binom{k^2}{k}$.
Assume that $n$ is divisible by $k^2$ and let $A$ be an $n \times n$ $P$-avoiding binary matrix.
Applying the second part of Lemma~\ref{lem:compression} with $t=k^2$ and $s=k$, and then the first part again with $t=k^2$ and $s=k$,
we get
\begin{align*}
\ex_P(n, n) 
&\le \frac{n}{k^2} \cdot g_{P^T}(k^2,k) + (k-1) \cdot \ex_P(n/k^2, n) \\
&\le \frac{n}{k^2} \cdot g_{P^T}(k^2,k) + (k-1) \cdot \frac{n}{k^2} \cdot g_P(k^2,k)  + (k-1)^2 \cdot \ex_P(n/k^2, n/k^2) \\
&\le n k^2 \cdot \binom{k^2}{k} + (k-1)^2 \cdot \ex_P(n/k^2, n/k^2).
\end{align*}
Solving this recurrence gives
\[
\ex_P(n,n) \le nk^3 \cdot \binom{k^2}{k}.
\]

\subsection{Outline of the proof}

In the proof of Theorem~\ref{thm:swtofh}, we will use height and width compression.
However, instead of using a single bound $f_P(k^2,k)$, we will refine the analysis by splitting the rows into several groups according to their number of $1$-entries, and use a more precise bound $f_P(a, 2^i b)$  on the number of rows in $i$th such group.
We will also use different pairs of the values $a,b$ in different stages of the proof.

We prepare all the necessary upper bounds on $f_P(a, 2^i b)$ in Subsection~\ref{subsec:swtofhbase},
combine them to two bounds on $g_P(a, b)$ in Subsection~\ref{subsec:swtofhdensenarrow}
and finish the proof in Subsection~\ref{subsec:swtofhinduction}.

\subsection{Upper bounds on the number of dense rows in a narrow matrix}
\label{subsec:swtofhbase}

Lemma~\ref{lem:swtofhContractionEnabler} is a generalization of a lemma from the proof of the earlier
upper bound on $c_P$ in terms of $s_P$~\cite[Lemma 3]{Cibulka09}. In the rest of this section, the symbol $e$ always stands for Euler's number.

\begin{lemma}
\label{lem:swtofhContractionEnabler}
Let $a$ and $b$ be integers satisfying $a\ge b \ge 1$.
If $b^2 \ge ae^2 s_P$, then
\[
f_P(a,b) < \frac{ae^2 s_P}{b}.
\]
\end{lemma}
\begin{proof}

Let $d = \lceil ae^2 s_P / b \rceil$.
Suppose, for a contradiction, that $B$ is a $P$-avoiding $d \times a$ matrix containing at least $b$ $1$-entries in each row. We show that $B$ contains occurrences of too many distinct $d$-permutation matrices.

We first estimate the total number of occurrences of $d$-permutation matrices in $B$.
We start by choosing one of the first $b$ $1$-entries in the first row and continue through all the rows,
always choosing a $1$-entry in a column from which no $1$-entry has been chosen.
Since $d \le b$, the total number of occurrences of $d$-permutation matrices in $B$ is at least $b!/(b-d)!$.

The number of occurrences of a fixed $d$-permutation matrix in $B$ is at most the number of ways to choose
a set of $d$ columns of $B$, which is $\binom{a}{d}$.
Every permutation matrix contained in $B$ avoids $P$.
Thus, a lower bound on the number of distinct $d$-permutation matrices occurring in $B$ is also a lower bound 
on the size of $S_P(d)$, that is,
\[
|S_P(d)| \ge \frac{b!/(b-d)!}{\binom{a}{d}} > \frac{(b/e)^d}{(ae/d)^d} = \left(\frac{bd}{ae^2}\right)^d.
\]
By the supermultiplicativity of $|S_P(d)|$, we have
\[
s_P = \lim_{n \to \infty}|S_P(nd)|^{1/(nd)} \ge |S_P(d)|^{1/d} > \frac{bd}{ae^2} \ge s_P,
\]
a contradiction.
\end{proof}

The following lemma gives upper bounds on $f_P(a,b)$ for a broad spectrum of densities $b/a$. 
For large densities, we apply Lemma~\ref{lem:swtofhContractionEnabler} directly.
For smaller densities, we combine two iterations of Lemma~\ref{lem:swtofhContractionEnabler} with height compression.
This is similar to the different treatment of wide and very wide blocks in~\cite{Cibulka09}.

\begin{definition}
\label{def:tightTriple}
We call an ordered triple $(b_0, b_1, b_2)$ of positive integers a \emph{tight triple}
if 
\begin{itemize}
\item $b_0 \ge b_1 \ge b_2$,
\item $b_1^2 > b_0 e^2 s_P$, \ and
\item $b_2^2 > b_1 e^2 s_P$.
\end{itemize}
\end{definition}

\begin{lemma}
\label{lem:swtofhRowBoundGeneral}
Let $(b_0, b_1, b_2)$ be a tight triple.
Then
\begin{enumerate}
\item[{\rm 1)}]
for every $i \in \left\{0,1,\ldots, \left\lfloor \log \left(b_0/b_1\right) \right\rfloor \right\}$, we have
\[
f_P(b_0, 2^i b_1) < b_1/2^i, \ \text{ and}
\]
\item[{\rm 2)}]
for every $i \in \left\{0,1,\ldots, \left\lfloor \log \left(b_1/b_2\right) \right\rfloor \right\}$,
we have
\[
f_P(b_0, 2^i b_2) < b_1 \lceil b_2/2^i \rceil \le b_1 b_2 / 2^{i-1}.
\]
\end{enumerate}
\end{lemma}

\begin{proof}
First we prove part 1). Since 
$i \le \left\lfloor \log \left(b_0/b_1\right) \right\rfloor$, we have $2^i b_1 \le b_0$.
By Definition~\ref{def:tightTriple}, the values $a=b_0$ and $b=2^ib_1$ satisfy the conditions
of Lemma~\ref{lem:swtofhContractionEnabler} and we obtain
\[
f_P(b_0, 2^i b_1) < b_0 e^2 s_P / (2^i b_1) < b^2_1 / (2^i b_1) = b_1 / 2^i .
\]

Now we prove part 2).
Let $B$ be a $P$-avoiding matrix
with $b_0$ columns and with at least $2^i b_2$ $1$-entries in every row. 
We split the rows of $B$ into blocks of $\lceil b_2 / 2^i\rceil$ consecutive rows.
If a block has all its $1$-entries in at most $b_1$ columns, then
these columns of the block induce a $\lceil b_2 / 2^i\rceil \times b_1$
matrix with $b_2 2^i$ $1$-entries in every row. Using part 1) for the triple $(b_1,b_2,b_2)$ we get $f_P(b_1, 2^i b_2) < b_2/2^i$, which is a contradiction.
Thus every block has at least $b_1$ nonzero columns and height
compression of the block creates a row with at least $b_1$ $1$-entries.
By part 1)
with $i=0$, there are less than $b_1$ such rows in the compressed matrix and thus
$B$ has less than $b_1 \lceil b_2/2^i\rceil$ rows. Therefore,
\[
f_P(b_0, b_2 2^i) < b_1 \lceil b_2/2^i\rceil .
\]

To prove the second inequality, we verify that $b_2 / 2^i>1$. Indeed, since $2^i \le b_1/b_2$, we have
\[
b_2 / 2^i \ge b^2_2 / b_1 > e^2 s_P > 1.
\]
This implies $\lceil b_2/2^i \rceil < 2 \cdot (b_2 / 2^i).$
\end{proof}

\subsection{Upper bounds on the number of 1-entries in a dense narrow matrix}
\label{subsec:swtofhdensenarrow}

By grouping the rows of a matrix according to their density and applying
Lemma~\ref{lem:swtofhRowBoundGeneral} to each such group, we get the following upper bounds
on the number of $1$-entries in the matrix.

\begin{lemma}
\label{lem:swtofhWideRowOnesCount}
Let $(b_0, b_1, b_2)$ be a tight triple.
Then
\begin{enumerate}
\item[{\rm 1)}]
$g_P(b_0,b_1) \le 2b_1^2 \cdot \log(2 b_0)$ and
\item[{\rm 2)}]
$g_P(b_0,b_2) \le 5b_1 b_2^2 \cdot \log(2 b_0)$.
\end{enumerate}
\end{lemma}

\begin{proof}
Let $B$ be a binary $P$-avoiding matrix with $b_0$ columns. 
For every $i \in \left\{0,1,\ldots, \left\lfloor \log \left(b_0/b_1\right) \right\rfloor \right\}$,
$B$ contains at most $f_P(b_0, 2^i b_1)$ rows with at least $2^i b_1$
(and less than $2^{i+1} b_1$) $1$-entries.
Thus, by the first part of Lemma~\ref{lem:swtofhRowBoundGeneral},
\begin{align*}
g_P(b_0,b_1) &\le
\sum_{i=0}^{\lfloor \log(b_0/b_1) \rfloor} f_P(b_0, 2^i b_1) \cdot 2^{i+1} b_1
\le \sum_{i=0}^{\lfloor \log(b_0/b_1) \rfloor} \frac{b_1}{2^i} \cdot 2^{i+1} b_1
= \sum_{i=0}^{\lfloor \log(b_0/b_1) \rfloor} 2 b_1^2 \\
& \le 2 b_1^2 \cdot \log(2 b_0).
\end{align*}

For every $i \in \left\{0,1,\ldots, \left\lfloor \log \left(b_1/b_2\right) \right\rfloor \right\}$,
$B$ contains at most $f_P(b_0, 2^i b_2)$ rows with at least $2^i b_2$
(and less than $2^{i+1} b_2$) $1$-entries.
Thus, by the second part of Lemma~\ref{lem:swtofhRowBoundGeneral},
\begin{align*}
g_P(b_0,b_2) - g_P(b_0,b_1) 
&\le \sum_{i=0}^{\lfloor \log(b_1/b_2) \rfloor} f_P(b_0, 2^i b_2) \cdot 2^{i+1} b_2
\le \sum_{i=0}^{\lfloor \log(b_1/b_2) \rfloor} \frac{b_1 b_2}{2^{i-1}} \cdot 2^{i+1} b_2 \\
&= \sum_{i=0}^{\lfloor \log(b_1/b_2) \rfloor} 4 b_1 b_2^2 
\le  4 b_1b_2^2 \cdot \log(2 b_1/b_2) \le 4 b_1b_2^2 \cdot \log(2 b_0).
\end{align*}

That is,
\[
g_P(b_0,b_2) \le 4 b_1b_2^2 \cdot \log(2 b_0) + g_P(b_0,b_1).
\]

Since $(b_0, b_1, b_2)$ is a tight triple, we have $b_1 \le b_2^2 / (e^2 s_P) \le b_2^2 / 2$.
Therefore
\begin{align*}
g_P(b_0,b_1) &\le b_1 b_2^2 \cdot \log(2 b_0) \qquad \text{and so} \\
g_P(b_0,b_2) &\le 4 b_1 b_2^2 \cdot \log(2 b_0) + b_1 b_2^2 \cdot \log(2 b_0) = 5b_1 b_2^2 \cdot \log(2 b_0).
\qedhere
\end{align*}
\end{proof}

\subsection{Upper bounds on the extremal function for rectangular matrices}
\label{subsec:swtofhinduction}

Let $\ex_P(m,n)$ be the maximum number of $1$-entries in an $m \times n$ binary matrix that avoids $P$.
If $x,y$ are positive real numbers, we let $\ex_P(x,y) = \ex_P(\lfloor x \rfloor, \lfloor y \rfloor)$. Clearly, the function $\ex_P: \mathbb{R}^+ \times \mathbb{R}^+ \rightarrow \mathbb{N} \cup \{0\}$ is nondecreasing in each coordinate, 
and satisfies the following trivial inequality for all $x,y>0$:
\begin{equation}
\label{eq_zaokrouhleni1}
\ex_P(x, \lceil y \rceil) \le \ex_P(x, y ) + x. 
\end{equation}
In the rest of the section, we will regard $m$ and $n$ as real variables.

We now prove two recursive formulas for the extremal function $\ex_P$, combining width compression and Lemma~\ref{lem:swtofhWideRowOnesCount}.

\begin{lemma}
\label{lem:swtofhGeneralInductionStep}
Let $m, n \ge 1$ be real numbers and let $(b_0, b_1, b_2)$ be a tight triple.
We have
\begin{enumerate}
\item[{\rm 1)}] $\ex_P(m, n) \le b_1\cdot \ex_P(m, n / b_0) + 2\cdot(b_1^2 / b_0) \cdot \log(2 b_0) \cdot n + b_1 m$ and
\item[{\rm 2)}] $\ex_P(m, n) \le b_2\cdot \ex_P(m, n / b_0) + 5\cdot(b_1 b_2^2 / b_0) \cdot \log(2 b_0) \cdot n + b_2 m$.
\end{enumerate}
\end{lemma}

\begin{proof}
To prove the first part, we use Lemma~\ref{lem:compression} with $t=b_0$ and $s=b_1$ and apply the upper bound on $g_P(b_0,b_1)$
from the first part of Lemma~\ref{lem:swtofhWideRowOnesCount}: 
\begin{align*}
\ex_P(m, n) \le \ex_P(\lfloor m \rfloor, b_0\lceil n/b_0\rceil)
&\le b_1 \cdot\ex_P(m, \lceil n / b_0 \rceil) + \frac{n}{b_0}\cdot g_P(b_0,b_1) \\
&\le b_1 \cdot\ex_P(m, \lceil n / b_0 \rceil) + 2\cdot\frac{b_1^2}{b_0} \cdot \log(2 b_0) \cdot n \\
&\le b_1 \cdot\ex_P(m, n / b_0) + 2\cdot\frac{b_1^2}{b_0} \cdot \log(2 b_0) \cdot n + b_1 m 
\qquad \text{by \eqref{eq_zaokrouhleni1}}.
\end{align*}

To prove the second part, we use Lemma~\ref{lem:compression} with $t=b_0$ and $s=b_2$
and apply the upper bound on $g_P(b_0,b_2)$ from the second part of Lemma~\ref{lem:swtofhWideRowOnesCount} to obtain
\begin{align*}
\ex_P(m, n) &\le \ex_P(\lfloor m \rfloor,b_0\lceil n/b_0\rceil)
\le b_2 \cdot\ex_P(m, \lceil n / b_0 \rceil) + 5\cdot\frac{b_1 b_2^2}{b_0} \cdot \log(2 b_0) \cdot n \\
&\le b_2 \cdot\ex_P(m, n / b_0) + 5\cdot\frac{b_1 b_2^2}{b_0} \cdot \log(2 b_0) \cdot n + b_2 m
\qquad \text{by \eqref{eq_zaokrouhleni1}}.
\qedhere
\end{align*}

\end{proof}

We now briefly sketch the strategy of the proof of Theorem~\ref{thm:swtofh}. We will be sequentially applying Lemma~\ref{lem:swtofhGeneralInductionStep}. 
For a given $b_0$, we always use the smallest values of $b_1$ and $b_2$ such that $(b_0, b_1, b_2)$
is a tight triple.
From the two bounds in Lemma~\ref{lem:swtofhGeneralInductionStep} we always choose the one that gives the smaller upper bound for the given parameters. 
We restrict ourselves to sequences of applications of Lemma~\ref{lem:swtofhGeneralInductionStep}
that alternate between width and height compression.
Thus it remains to choose the sequence of values of the compression factors $b_0$.
We optimize among sequences where we use one value of $b_0$ for the first application of the lemma
and a second value for all the subsequent applications.
We found that the optimal value of $b_0$ for the first application is $\Theta(s_P^4)$,
and for the subsequent ones $\Theta(s_P^5)$.
In other words, in the proof of Theorem~\ref{thm:swtofh}, we first apply width $\Theta(s_P^4)$-compression to the given $n\times n$ matrix, and then we will alternate height $\Theta(s_P^5)$-compression and width $\Theta(s_P^5)$-compression. During the process, we will be obtaining \emph{thin} matrices with height/width ratio $\Theta(s_P^4)$, and \emph{wide} matrices with height/width ratio $\Theta(s_P^{-1})$.

We derive upper bounds on the extremal function $ex_P$ of square, thin and wide matrices recursively.
First, in Lemma~\ref{lem:swtofh:indukcniKrok} we use the first part of
Lemma~\ref{lem:swtofhGeneralInductionStep} to derive an upper bound on $ex_P$ for thin matrices in terms of the extremal function for smaller wide matrices.
Similarly, we use the second part of Lemma~\ref{lem:swtofhGeneralInductionStep} to obtain an upper bound on $ex_P$ for wide matrices in terms of the extremal function for smaller thin matrices.

In Corollary~\ref{cor:swtofh:rectangular}, we derive explicit upper bound on $ex_P$ for thin matrices.

Finally, by combining Lemma~\ref{lem:swtofhGeneralInductionStep} with Corollary~\ref{cor:swtofh:rectangular}, we obtain an upper bound on $ex_P$ for square matrices.

We now proceed with detailed proofs.
We computed the following parameters that we substitute for $(b_0,b_1,b_2)$.

\begin{observation}
\label{obs:tightTriples}
The triples $(\lceil e^{19} s_P^{5} \rceil, \lceil e^{11} s_P^{3} \rceil,
\lceil e^7 s_P^{2} \rceil)$ and $(\lceil e^{3} s_P^{4} \rceil, \lceil e^{3} s_P^{2.5} \rceil,
\lceil e^3 s_P^{1.75} \rceil)$ are tight triples when $s_P$ is large enough.
\end{observation}

\begin{lemma}
\label{lem:swtofh:indukcniKrok}
We have
\begin{enumerate}
\item[{\rm 1)}]
$\ex_P(m, s^{-4}_P\cdot m) \le e^{11.1} s^3_P \cdot \ex_P(e^{-19}s^{-5}_P\cdot m, s^{-4}_P\cdot m) + O(s_P \log(s_P)) \cdot m.$
\item[{\rm 2)}]
$\ex_P(e^{-19}s^{-1}_P\cdot m, m) \le e^{7.1} s^2_P \cdot \ex_P(e^{-19}s^{-1}_P\cdot m, e^{-19} s^{-5}_P\cdot m) + O(s_P^2 \log(s_P)) \cdot m.$
\end{enumerate}
\end{lemma}

\begin{proof}
Let $b_0 = \lceil e^{19} s_P^{5} \rceil$, $b_1 = \lceil e^{11} s_P^{3} \rceil$
and $b_2 = \lceil e^7 s_P^{2} \rceil$.
By Observation~\ref{obs:tightTriples}, $(b_0, b_1, b_2)$ is a tight triple, and so we can use
Lemma~\ref{lem:swtofhGeneralInductionStep}.
When $s_P$ is large enough, we have $b_1 \le e^{11.1} s_P^3$ and $b_2 \le e^{7.1} s_P^2$. 
By the first part of Lemma~\ref{lem:swtofhGeneralInductionStep}, we have

\begin{align*}
\ex_P(s^{-4}_P\cdot m, m) 
&\le b_1 \cdot \ex_P(s^{-4}_P\cdot m, e^{-19}s^{-5}_P\cdot m) + \frac{2 b_1^2}{b_0} \log(2 b_0) \cdot m + b_1 s^{-4}_P \cdot m \\
&\le e^{11.1} s^3_P \cdot \ex_P(s^{-4}_P\cdot m, e^{-19}s^{-5}_P\cdot m) + O\left(s_P \log(s_P)\right) \cdot m.
\end{align*}

Since $s_P = s_{P^T}$, this implies
\[
\ex_P(m, s^{-4}_P\cdot m) 
\le e^{11.1} s^3_P \cdot \ex_P(e^{-19}s^{-5}_P\cdot m, s^{-4}_P\cdot m) + O\left(s_P \log(s_P)\right) \cdot m.
\]

By the second part of Lemma~\ref{lem:swtofhGeneralInductionStep}, we have
\begin{align*}
\ex_P(e^{-19}s^{-1}_P\cdot m, m) 
&\le b_2 \cdot \ex_P(e^{-19}s^{-1}_P\cdot m, e^{-19} s^{-5}_P\cdot m) + \frac{5 b_1 b_2^2}{b_0} \cdot \log(2 b_0) \cdot m \\ &\quad + b_2 \cdot e^{-19} s^{-1}_P \cdot m\\
&\le e^{7.1} s^2_P \cdot \ex_P(e^{-19}s^{-1}_P\cdot m, e^{-19} s^{-5}_P\cdot m) + O\left(s^2_P \log(s_P)\right) \cdot m.
\qedhere
\end{align*}
\end{proof}

\begin{corollary}
\label{cor:swtofh:rectangular}
We have 
$
\ex_P(n, s^{-4}_P\cdot n) \le O(s_P \log(s_P)) \cdot n.
$
\end{corollary}

\begin{proof}
Combining the first and second part of Lemma~\ref{lem:swtofh:indukcniKrok} with $m=n$ and $m=s^{-4}_P \cdot n$, respectively, 
we have
\begin{align}
\ex_P(n, s^{-4}_P\cdot n) &\le e^{11.1} s^3_P \cdot \ex_P(e^{-19}s^{-5}_P\cdot n, s^{-4}_P\cdot n) + O(s_P \log(s_P)) \cdot n. \\
\label{eq_thin}
&\le e^{18.2} s^5_P \cdot \ex_P(e^{-19}s^{-5}_P\cdot n, e^{-19} s^{-9}_P\cdot n) + O(s_P \log(s_P)) \cdot n.
\end{align}

Let $c_1>0$ be a constant such that for all $n\ge 1$, the second term in \eqref{eq_thin} is bounded from above by $c_1\cdot s_P \log(s_P)\cdot n$. Solving the recursion, we get
\[
\ex_P(n, s^{-4}_P\cdot n) \le c_1\cdot s_P \log(s_P) \cdot n \cdot  \sum_{k=0}^{\infty} e^{-0.8k} \le O(s_P \log(s_P)) \cdot n.
\qedhere
\]
\end{proof}

We are now ready to finish the proof of Theorem~\ref{thm:swtofh}.

\begin{proof}[Proof of Theorem~\ref{thm:swtofh}]
Combining the second part of Lemma~\ref{lem:swtofhGeneralInductionStep} with $b_0 = \lceil e^3 s_P^{4} \rceil$, $b_1 = \lceil e^3 s_P^{2.5} \rceil$ and $b_2 = \lceil e^3 s_P^{1.75} \rceil$, and Corollary~\ref{cor:swtofh:rectangular}, we get
\begin{align*}
\ex_P(n, n) 
&\le \ex_P(n, e^{3} n) \\
&\le b_2 \cdot \ex_P(n, s^{-4}_P n) + \frac{5 b_1 b_2^2}{b_0} \cdot \log(2b_0) \cdot n + b_2 n \\
&\le O(s_P^{1.75} \cdot s_P \log(s_P) \cdot n) + O(s_P^2 \log(s_P) \cdot n) + O(s^{1.75}_P n)  \\
&= O(s_P^{2.75} \log(s_P)) \cdot n.
\qedhere
\end{align*}
\end{proof}

\section{Higher-dimensional matrices}
\label{sec:higherdim}

Let $P$ be a given $d$-dimensional $k$-permutation matrix $P$. Recall that $S_P(n)$ 
denotes the set of all $d$-dimensional $n$-permutation matrices avoiding $P$. 
Let $T_P(n)$ be the set of all $d$-dimensional matrices of size $n \times \dots \times n$ that avoid $P$.

\subsection{Upper bound}
\label{subsec:HighDimSWUpper}

The proof of the upper bound in Theorem~\ref{thm:higherdim} uses the following high-dimensional 
generalization of a lemma by Fox~\cite[Lemma 11]{Fox13}, which he used in his simplified proof of the bound $s_P \in O(c_P)^2$.

\begin{lemma}
\label{lem:high-dim-fox-reduction}
Let $P$ be a $d$-dimensional permutation matrix.
Let $t,u \in \mathbb{N}$ and let $n=tu$. 
Then
\[
\lvert S_P(n)\rvert \le t^{dn} \lvert T_P(u) \rvert
\]
\end{lemma}

\begin{proof}
Let $A$ be a $d$-dimensional $n$-permutation matrix that avoids $P$.
We split $A$ into blocks of size $t \times t \times \cdots \times t$ by hyperplanes orthogonal
to the coordinate axes.
Let $B$ be the $u \times u \times \cdots \times u$ matrix formed by contracting these blocks.
The number of choices for $B$ is at most $T_P(u)$.

Let the $i$th \emph{slice} of a $d$-dimensional matrix be the set of entries with first coordinate equal to $i$.
Since $A$ is a permutation matrix, each slice of $A$ contains exactly one $1$-entry and each slice of
$B$ contains at most $t$ $1$-entries.

We fix one such $B$ and count the number of $d$-dimensional $n$-permutation matrices $A$ whose
contraction gives $B$.
For every $i$, the $1$-entry in the $i$th slice of $A$ can correspond to one of the at most $t$ $1$-entries
in the $\lfloor i/t \rfloor$th slice of $B$.
After selecting this $1$-entry of $B$, we have $t^{d-1}$ positions for the $1$-entry in the $i$th slice of $A$.
Hence, the number of $d$-dimensional $n$-permutation matrices $A$ whose contraction gives $B$ is at most $(t^d)^n$.
\end{proof}

\begin{proof}[Proof of the upper bound in Theorem~\ref{thm:higherdim}]
We assume without loss of generality that $n=2^m$ for some $m\in\mathbb{N}$. Our plan is to show an upper bound on $|T_P(u)|$ for a suitable $u=2^i$ and then use Lemma~\ref{lem:high-dim-fox-reduction}. Every $2^i \times \dots \times 2^i$  $d$-dimensional $P$-avoiding binary matrix can be built by a sequence of expansions from smaller $d$-dimensional $P$-avoiding matrices, reversing the contraction operation of $2\times \dots \times 2$ blocks.
We start with $A_0$, the $1\times \dots \times 1$ $d$-dimensional matrix containing one $1$-entry.
In each step, we transform the matrix $A_i$ of size $2^i \times \dots \times 2^i$ 
into a matrix $A_{i+1}$ of size $2^{i+1} \times \dots \times 2^{i+1}$ by replacing each $0$-entry 
of $A_i$ by a $2 \times \dots \times 2$ block containing only $0$-entries and each $1$-entry
of $A_i$ by a $2 \times \dots \times 2$ block containing at least one $1$-entry. There is a 
single possibility of replacing a $0$-entry and $2^{2^d}-1$ possibilities of replacing a 
$1$-entry.

We use the high-dimensional generalization of the F\"{u}redi--Hajnal conjecture, that is, 
the estimate $\ex_P(n) = \Theta_{P,d}(n^{d-1})$~\cite{KlazarMarcus}. 
Thus, $\ex_P(2^i) \le c_P 2^{i(d-1)}$ for some constant $c_P$ and so
\[ |T_P(2^i)|
\le 2^{2^d \cdot c_P \cdot 2^{(i-1)(d-1)}} \cdot |T_P({2^{i-1}})|
\le \cdots \le 2^{2^d \cdot c_P \cdot 2^{i(d-1)}}.
\]

We select 
\[
i=\left\lfloor \frac{1}{d-1} \cdot \log\left(\frac{n}{2^d \cdot c_P}\right) \right\rfloor,
\] 
so that $|T_P({2^i})| \le 2^{n}$. We have
\[
2^{i} \ge \frac{1}{2} \cdot \left(\frac{n}{2^d \cdot c_P}\right)^{1/(d-1)} 
= \left(\frac{n}{2^{2d-1} \cdot c_P}\right)^{1/(d-1)}.
\]

By Lemma~\ref{lem:high-dim-fox-reduction}
with $u=2^i$ and $t=n/2^i$, we have
\begin{align*}
|S_P(n)| &\le t^{dn} |T_P(2^i)| \le n^{dn} \cdot 2^{-idn} \cdot 2^n
\le n^{dn} \cdot (2^{2d-1} \cdot c_P / n)^{dn/(d-1)} \cdot 2^{n} \\
&\le n^{dn(1-1/(d-1))} \cdot (2^{2d-1} \cdot c_P)^{dn/(d-1)} \cdot 2^{n} \\
& \le (n^n)^{d-d/(d-1)} \cdot \left(2^{2d} \cdot c_P\right) ^{d n/(d-1)}
\qquad \text{ and, using } n! \ge \left(\frac{n}{e}\right)^n, \\
& \le (e^n n!)^{d-1-1/(d-1)} \cdot \left(2^{2d} \cdot c_P\right) ^{d n/(d-1)} \\
& \le (n!)^{d-1-1/(d-1)} \cdot e^{dn} \cdot \left(2^{2d} \cdot c_P\right)^{d n/(d-1)}.
\end{align*}

Using the upper bound $c_P \le 2^{O_d(k)}$ implied by the result of Geneson and Tian~\cite[Equation (4.5)]{GenesonTian}, we obtain
\[
|S_P(n)| \le \left(2^{O_d(k)}\right)^n (n!)^{d-1-1/(d-1)}. \qedhere
\]
\end{proof}

\subsection{Lower bound}
\label{subsec:HighDimSWLower}

A partial order $\prec$ on $[n]$ is an \emph{intersection} of $d$ linear orders $<_1$, $<_2$, \dots, $<_d$ on $[n]$ if $\forall a, b \in [n]\ (a \prec b\Leftrightarrow \forall i \in [d]\ \  a <_i b)$.
A partial order $\prec$ has \emph{dimension} $d$ if $d$ is the smallest positive integer such that $\prec$ is an intersection of $d$ linear orders.
A \emph{random $d$-dimensional partial order} on $[n]$ is the intersection of $d$ linear orders on $[n]$ taken 
uniformly and independently at random. 
A partial order is an \emph{antichain} if no two elements are comparable by the partial order.
A linear order $<$ on $[n]$ is a linear extension of $\prec$ if $\forall a,b \in [n]\  (a \prec b \Rightarrow a<b$).

Brightwell~\cite{Brightwell92} showed the following lower bound on the number of linear extensions of almost all partial orders of a given dimension.

\begin{theorem}[{Brightwell~\cite[Corollary 4]{Brightwell92}}]
\label{thm:brightwell}
Almost every $(d-1)$-dimensional partial order on $[n]$ has at least $\left(e^{-2} n^{1-1/(d-1)}\right)^n$ 
linear extensions.
\end{theorem}

Let $Q_d(n)$ be the probability that a random $d$-dimensional partial order on $[n]$ is an antichain.

\begin{corollary}
\label{cor_antiretezece}
We have
\[
Q_d(n) \ge \frac{\left(e^{-2} n^{1-1/(d-1)}\right)^n}{n!}.
\]
\end{corollary}
\begin{proof}
The \emph{reverse} $>$ of a linear order $<$ on $[n]$ is the linear order satisfying 
for every distinct $a,b$ from $[n]$ that $a>b$ if and only if $b<a$.
By a well-known observation (see e.g.\ the introduction of the Brightwell's paper~\cite{Brightwell92}), 
a linear order $<$ is a linear extension of $\prec$ if and only if the intersection of $\prec$ and 
the reverse of $<$ is an antichain.
Therefore, the expected number of linear extensions of a random $(d-1)$-dimensional partial order is
$n!$ times larger than $Q_d(n)$.
\end{proof}

Let $I^d_k$ be the $d$-dimensional $k$-permutation matrix with $1$-entries at positions
$(i,i,\ldots, i)$ for every $i \in [k]$.

\begin{theorem}
We have
\[
|S_{I^d_2}(n)| \ge (1/(en))\cdot e^{-(1+1/(d-1))n} \cdot (n!)^{d-1-1/(d-1)}.
\]
\end{theorem}
\begin{proof}
We consider the uniform probability space of $d$-dimensional $n$-permutation matrices, that is, each of the 
$(n!)^{d-1}$ matrices has probability $1/(n!)^{d-1}$.
A random $d$-dimensional $n$-permutation matrix from this space can be formed by taking $d$ permutations
$\pi_1,\allowbreak \pi_2, \dots, \pi_d$ of $[n]$ independently and uniformly at random, and placing $1$-entries 
to positions $(\pi_1(a),\allowbreak \pi_2(a), \dots, \pi_d(a))$ for every $a \in [n]$.

Consider a $d$-dimensional $n$-permutation matrix $R$ with $1$-entries at positions 
$(\pi_1(a),\allowbreak \pi_2(a), \dots, \pi_d(a))$, for every $a \in [n]$.
We define the partial order $\prec_R$ as the intersection of the linear orders $<_1,\allowbreak <_2, \dots, <_d$ where $a <_i b$ 
if and only if $\pi_i(a) < \pi_i(b)$.
Thus, if $R$ is a random $d$-dimensional $n$-permutation matrix, then $\prec_R$ is a random $d$-dimensional 
partial order on $[n]$.
An occurrence of $I^d_2$ in $R$ corresponds to a pair of elements of $[n]$ comparable in $\prec_R$, and so
$R$ avoids $I^d_2$ if and only if $\prec_R$ is an antichain.

Consequently, by Corollary~\ref{cor_antiretezece}, the probability that a random $d$-dimensional $n$-permutation matrix 
avoids $I^d_2$ is at least
$\bigl(e^{-2} n^{1-1/(d-1)}\bigr)^n/n!$ and thus 
\begin{align*}
|S_{I^d_2}(n)| &\ge e^{-2n} \cdot n^{n(1-1/(d-1))} \cdot (n!)^{d-2} 
\ge e^{-2n} \cdot (n!\cdot e^n /(en))^{1-1/(d-1)} \cdot (n!)^{d-2} \\
&\ge (1/(en))\cdot e^{-(1+1/(d-1))n} \cdot (n!)^{d-1-1/(d-1)}.
\qedhere
\end{align*}
\end{proof}

\begin{theorem}
\label{thm:HighDimSWLowerDiagonal}
We have
\[
|S_{I^d_k}(n)| \ge n^{-O_d(k)}  \left(\Omega_d(k^{1/(d-1)})\right)^n \cdot (n!)^{d-1-1/(d-1)},
\]
where the constants hidden by $\Omega$ and $O$ do not depend on $n$ and $k$.
\end{theorem}
\begin{proof}
All permutations in this proof are $d$-dimensional.
Let $A$ be an $lm$-permutation matrix.
If we can split the $1$-entries of $A$ into $l$ $m$-tuples such that each of these $m$-tuples forms an occurrence of 
an $I^d_2$-avoiding matrix, then $A$ avoids $I^d_{l+1}$.
We now count how many permutation matrices we obtain by the reverse process, that is, by merging 
$l$ $I^d_2$-avoiding $m$-permutation matrices to form an $I^d_{l+1}$-avoiding matrix.

The number of ways to choose an ordered $l$-tuple of matrices from $S_{I^d_2}(m)$ is 
\[
|S_{I^d_2}(m)|^l.
\]
Given an $l$-tuple of $m$-permutation matrices, the number of ways to form an $lm$-permutation matrix
whose $1$-entries can be split into $l$ $m$-tuples forming the occurrences of the $l$ selected permutations is
\[
\binom{lm}{m,m,\ldots,m}^d.
\]
The number of ways to split the $1$-entries of an $lm$-permutation matrix into $l$ $m$-tuples is
\[
\binom{lm}{m,m,\ldots,m}.
\]

We thus have
\begin{align*}
|S_{I^d_{l+1}}(lm)| &\ge |S_{I^d_2}(m)|^l \cdot \binom{lm}{m,m,\ldots,m}^d \cdot \binom{lm}{m,m,\ldots,m}^{-1} \ge \\
& \ge \left(\frac{1}{em}\right)^l\cdot e^{-(1+1/(d-1))lm} \cdot (m!)^{(d-1-1/(d-1))l} \cdot \left(\frac{(lm)!}{(m!)^l}\right)^{d-1} = \\
& = \left(\frac{1}{em}\right)^l\cdot \frac{1}{e^{(1+1/(d-1))lm}} \cdot  \left(\frac{(lm)!}{(m!)^l}\right)^{1/(d-1)} \cdot (lm!)^{d-1-1/(d-1)}.
\end{align*}

We have
\[
\frac{(lm)!}{(m!)^l} \ge \frac{(lm/e)^{lm}}{(em(m/e)^m)^l} = \frac{l^{lm}}{(em)^l}
\]
and so
\begin{align*}
|S_{I^d_{l+1}}(lm)| 
&\ge \left(\frac{1}{em}\right)^l\cdot \frac{1}{e^{(1+1/(d-1))lm}} \cdot  \left(\frac{l^{lm}}{(em)^l}\right)^{1/(d-1)}\cdot (lm!)^{d-1-1/(d-1)} \\
&\ge \left(\frac{1}{em}\right)^{2l}\cdot \left(\frac{l^{1/(d-1)}}{e^{1+1/(d-1)}}\right)^{lm}\cdot (lm!)^{d-1-1/(d-1)}.
\end{align*}
That is, when $k = l+1$ and $n = lm = (k-1) m$, we have
\[
|S_{I^d_k}(n)| \ge \left(\frac{1}{en}\right)^{2k}\cdot \left(\frac{(k-1)^{1/(d-1)}}{e^{1+1/(d-1)}}\right)^n\cdot (n!)^{d-1-1/(d-1)}.
\qedhere
\]
\end{proof}

\begin{proof}[Proof of the lower bound in Theorem~\ref{thm:higherdim}]
A $d$-dimensional permutation matrix $M$ is \emph{monotone} if its $1$-entries can be ordered in such a way that 
for every $i \in [d]$, the $i$th coordinates of the $1$-entries are either increasing or decreasing. 
Observe that by symmetry, $|S_M(n)| = |S_{I^d_k}(n)|$ for every $n,k \in \mathbb{N}$ 
and every monotone $k$-permutation matrix $M$.
By applying the Erd\H{o}s--Szekeres lemma on monotone subsequences~\cite{ErdosSzekeres} $d-1$ times,
every $d$-dimensional $k$-permutation
matrix $P$ contains a monotone $d$-dimensional $\lceil k^{1/2^{d-1}} \rceil$-permutation
(see also~\cite{Kruskal53}).
Therefore the lower bound in Theorem~\ref{thm:higherdim} is a corollary of Theorem~\ref{thm:HighDimSWLowerDiagonal}.
\end{proof}


\section{Concluding remarks}
\label{sec_conclusion}
\subsection{Specific permutation matrices}
There are two types of permutation matrices $P$ for which we have a subexponential upper bound on their F\"uredi--Hajnal limit $c_P$. The first type are the scattered matrices, which have generally very little structure. The second type includes practically all previously known examples of matrices with subexponential F\"uredi--Hajnal limit, and consists of matrices obtained from the identity matrix by a few elementary operations, like the direct sum. The \emph{direct sum} of a $k\times k$ matrix $A$ and an $l\times l$ matrix $B$ is the $(k+l)\times(k+l)$ block matrix $\left(\begin{smallmatrix}0&B\\A&0\end{smallmatrix}\right)$. Similarly, the \emph{skew sum} of $A$ and $B$ is the block matrix $\left(\begin{smallmatrix}A&0\\0&B\end{smallmatrix}\right)$. \emph{Layered} matrices, obtained as a multiple direct sum of identity matrices, form the most natural class for which a polynomial upper bound on $c_P$ is known. The upper bound follows from the upper bound $s_P \le 4 k^2$ on the Stanley--Wilf limit of every layered $k$-permutation $P$~\cite{CJS12}, since $c_P$ and $s_P$ are polynomially related~\cite{Cibulka09}. More recently, the first author~\cite{Cibulka13} has shown directly that $c_P$ is at most linear in $k$ for every layered $k$-permutation $P$.
The matrices of the second type have generally a lot of structure; in particular, they are far from being scattered. 
We have added the cross matrix and certain grid products to the second type of matrices, but there are still many matrices that do not belong to any of these types. For example, we do not have any subexponential upper bound on the F\"uredi--Hajnal limit of a permutation matrix whose $1$-entries in the odd columns lie on the diagonal and even columns induce a scattered matrix.

The grid product of two permutation matrices is a special case of a binary matrix obtained by the following operation. Let $A$ be a $k\times k$ binary matrix and $B$ an $l\times l$ binary matrix. The \emph{Minkowski sum} of $A$ and $B$ is the $(k+l)\times(k+l)$ binary matrix with a $1$-entry at position $(i,j)$ if and only if there exist $i_1,i_2,j_1,j_2$ such that $A_{i_1,j_1}=1$, $B_{i_2,j_2}=1$, $i_1+i_2=i$ and $j_1+j_2=j$. So far we do not know any general subexponential bound on $c_P$ when $P$ is a permutation matrix contained in a Minkowski sum of $A$ and $B$ where $A$ is either a scattered permutation matrix or a permutation matrix with polynomial $c_A$, and $B$ is a matrix with just two $1$-entries. In fact, we do not know any general subexponential bound even for permutation matrices contained in \emph{two-diagonal} matrices; that is, binary matrices whose all $1$-entries lie on two parallel diagonals, which may be arbitrarily far apart.

\begin{question}
Is $c_P$ polynomial in $k$ for $k$-permutation matrices $P$ contained in a two-diagonal matrix?
\end{question}

\emph{Decomposable} permutation matrices generalize layered matrices and are defined as the smallest class of matrices closed under direct sum and skew sum, and containing all identity matrices. The cross matrix $\Cross_k$ is decomposable but our current upper bound on its F\"uredi--Hajnal limit is slightly superpolynomial.

\begin{question}
Is $c_{\Cross_k}$ polynomial in $k$? 
\end{question}

\subsection{Higher-dimensional matrices}
In the $2$-dimensional case, it was shown by Arratia~\cite{Arratia99} that the limit 
$\lim_{n\rightarrow \infty} |S_{P}(n)|^{1/n}$ exists for every permutation matrix $P$.
Analogously, by the super-additivity shown by Pach and Tardos~\cite{PachTardos}, the 
limit $\lim_{n\rightarrow \infty} \ex_P(n)/n$ always exists.
Geneson and Tian~\cite[Lemma 4.7]{GenesonTian} 
showed that for every $d>2$ and every $d$-dimensional 
permutation matrix $P$ with at least one $1$-entry in a corner, 
there exists a constant $K$ such that $\ex_P(sn) \ge Ks^{d-1} \ex_P(n)$ for every integer $n$.
They asked whether this holds with $K=1$ for every $P$. 
A positive answer would imply the existence of the F\"{u}redi--Hajnal limit 
$\lim_{n\rightarrow \infty} \ex_P(n)/n$ for every $d$-dimensional permutation matrix $P$.

We pose an analogous question about the higher-dimensional Stanley--Wilf limit.
\begin{question}
Does the limit 
\[
\lim_{n \rightarrow \infty} \left( \frac{|S_{P}(n)|}{(n!)^{d-1-1/(d-1)}} \right)^{1/n}
\]
exist for every $d>2$ and every $d$-dimensional permutation matrix $P$?
\end{question}

Let 
\begin{align*}
\overline{s}_P &= \limsup_{n \rightarrow \infty} 
\left( \frac{|S_{P}(n)|}{(n!)^{d-1-1/(d-1)}} \right)^{1/n} \quad \text{and}\\
\underline{s}_P &= \liminf_{n \rightarrow \infty} 
\left( \frac{|S_{P}(n)|}{(n!)^{d-1-1/(d-1)}} \right)^{1/n}.
\end{align*}

We have seen in Section~\ref{subsec:HighDimSWLower} that the number of $d$-dimensional
$n$-permutation matrices avoiding the $d$-dimensional $2$-permutation matrix $I^d_2$ 
with $1$-entries at $(1,1,\ldots, 1)$ and $(2,2,\ldots, 2)$ is equal to $(n!)^{d-1}$ times 
the probability that a random $d$-dimensional partial order on $[n]$ is an antichain.
Thus the following are the best known bounds on the limit superior and limit inferior
of $I^d_2$:
\begin{align*}
\underline{s}_{I^3_2} &\ge 1 \quad \text{\cite{Sidorenko91}} \\
\overline{s}_{I^3_2} &\le \sqrt{\pi/2} \quad \text{\cite{BBS99}} \\
e^{-1-1/(d-1)} \le \underline{s}_{I^d_2} \le \overline{s}_{I^d_2} 
&\le 2(d-1) e^{1-1/(d-1)} \quad \text{for $d\ge 4$~\cite{Brightwell92}}.
\end{align*}

In the general case, Theorem~\ref{thm:higherdim} gives the following.
For every $d,k \ge 2$ and every $d$-dimensional $k$-permutation matrix $P$,
\[
\Omega_d\left(k^{1/(2^d (d-1))}\right) \le \underline{s}_P \le \overline{s}_P \le 2^{O_d(k)}.
\]

From the proof of the upper bound of Theorem~\ref{thm:higherdim} in Section~\ref{subsec:HighDimSWUpper}
we know that $\overline{s}_P$ is bounded from above by $O_d((\overline{c}_P)^{d/(d-1)})$
for every $d$-dimensional permutation matrix.
In the case $d=2$, $c_P$ is also bounded from above by a polynomial in $s_P$~\cite{Cibulka09}, 
but it is not known whether the following is true.
\begin{question}
Is $\underline{c}_P$ bounded from above by a polynomial in $\overline{s}_P$ 
for all $d$-dimensional permutation matrices $P$?
\end{question}



\begin{thebibliography}{10}
\expandafter\ifx\csname url\endcsname\relax
  \def\url#1{{\tt #1}}\fi
\expandafter\ifx\csname urlprefix\endcsname\relax\def\urlprefix{URL }\fi

\bibitem{AERWZ06}
M.~H. Albert, M.~Elder, A.~Rechnitzer, P.~Westcott and M.~Zabrocki, On the
  {S}tanley--{W}ilf limit of 4231-avoiding permutations and a conjecture of
  {A}rratia, {\em Adv. in Appl. Math.\/} {\bf 36}(2) (2006), 96--105.

\bibitem{Arratia99}
R.~Arratia, On the {S}tanley-{W}ilf conjecture for the number of permutations
  avoiding a given pattern, {\em Electron. J. Combin.\/} {\bf 6} (1999), Note,
  N1, 4 pp. (electronic).

\bibitem{Bevan14}
D.~Bevan, Permutations avoiding 1324 and patterns in {{\L}}ukasiewicz paths,
  {\em J. Lond. Math. Soc. (2)\/} {\bf 92}(1) (2015), 105--122.

\bibitem{BBS99}
B.~Bollob{\'a}s, G.~Brightwell and A.~Sidorenko, Geometrical techniques for
  estimating numbers of linear extensions, {\em European J. Combin.\/} {\bf
  20}(5) (1999), 329--335.

\bibitem{Bona07Records}
M.~B{\'o}na, New records in {S}tanley--{W}ilf limits, {\em European J.
  Combin.\/} {\bf 28}(1) (2007), 75--85.

\bibitem{Brightwell92}
G.~Brightwell, Random {$k$}-dimensional orders: {W}idth and number of linear
  extensions, {\em Order\/} {\bf 9}(4) (1992), 333--342.

\bibitem{Cibulka09}
J.~Cibulka, On constants in the {F}{\"u}redi--{H}ajnal and the
  {S}tanley--{W}ilf conjecture, {\em J. Combin. Theory Ser. A\/} {\bf 116}(2)
  (2009), 290--302.

\bibitem{Cibulka13}
J.~Cibulka, Extremal combinatorics of matrices, sequences and sets of
  permutations (2013), {P}h.D. thesis, Charles University,
  \url{https://dspace.cuni.cz/handle/20.500.11956/59415}.

\bibitem{CJS12}
A.~Claesson, V.~Jel{\'{\i}}nek and E.~Steingr{\'{\i}}msson, Upper bounds for
  the {S}tanley--{W}ilf limit of 1324 and other layered patterns, {\em J.
  Combin. Theory Ser. A\/} {\bf 119}(8) (2012), 1680--1691.

\bibitem{ErdosSzekeres}
P.~Erd\H{o}s and G.~Szekeres, A combinatorial problem in geometry, {\em Compos.
  Math\/} {\bf 2} (1935), 463--470.

\bibitem{Fox13}
J.~Fox, {S}tanley--{W}ilf limits are typically exponential (2013),
  \href{https://arxiv.org/abs/1310.8378v1}{arXiv:1310.8378v1}.

\bibitem{FoxPersComm}
J.~Fox, personal communication (2016).

\bibitem{FurediHajnal}
Z.~F{\"u}redi and P.~Hajnal, Davenport--{S}chinzel theory of matrices, {\em
  Discrete Math.\/} {\bf 103}(3) (1992), 233--251.

\bibitem{GenesonTian}
J.~T. Geneson and P.~M. Tian, Extremal functions of forbidden multidimensional
  matrices (2015),
  \href{https://arxiv.org/abs/1506.03874v1}{arXiv:1506.03874v1}.

\bibitem{GuillemotMarx}
S.~Guillemot and D.~Marx, Finding small patterns in permutations in linear
  time, {\em Proceedings of the {T}wenty-{F}ifth {A}nnual {ACM}-{SIAM}
  {S}ymposium on {D}iscrete {A}lgorithms\/}, ACM, New York (2014) 82--101.

\bibitem{KaiserKlazar}
T.~Kaiser and M.~Klazar, On growth rates of closed permutation classes, {\em
  Electron. J. Combin.\/} {\bf 9}(2) (2002/03), Research paper 10, 20 pp.
  (electronic).

\bibitem{Klazar00}
M.~Klazar, The {F}\"uredi--{H}ajnal conjecture implies the {S}tanley--{W}ilf
  conjecture, {\em Formal power series and algebraic combinatorics ({M}oscow,
  2000)\/}, Springer, Berlin (2000) 250--255.

\bibitem{KlazarMarcus}
M.~Klazar and A.~Marcus, Extensions of the linear bound in the
  {F}\"uredi--{H}ajnal conjecture, {\em Adv. in Appl. Math.\/} {\bf 38}(2)
  (2007), 258--266.

\bibitem{Kruskal53}
J.~B. Kruskal, Jr., Monotonic subsequences, {\em Proc. Amer. Math. Soc.\/} {\bf
  4} (1953), 264--274.

\bibitem{MarcusTardos}
A.~Marcus and G.~Tardos, Excluded permutation matrices and the
  {S}tanley--{W}ilf conjecture, {\em J. Combin. Theory Ser. A\/} {\bf 107}(1)
  (2004), 153--160.

\bibitem{PachTardos}
J.~Pach and G.~Tardos, Forbidden paths and cycles in ordered graphs and
  matrices, {\em Israel J. Math.\/} {\bf 155} (2006), 359--380.

\bibitem{Regev81}
A.~Regev, Asymptotic values for degrees associated with strips of {Y}oung
  diagrams, {\em Adv. in Math.\/} {\bf 41}(2) (1981), 115--136.

\bibitem{Sidorenko91}
A.~Sidorenko, Inequalities for the number of linear extensions, {\em Order\/}
  {\bf 8}(4) (1991/92), 331--340.

\bibitem{Steingrimsson13:survey}
E.~Steingr{\'{\i}}msson, Some open problems on permutation patterns, {\em
  Surveys in combinatorics 2013\/}, vol. 409 of {\em London Math. Soc. Lecture
  Note Ser.\/}, Cambridge Univ. Press, Cambridge (2013) 239--263.

\end{thebibliography}
\end{document}